\documentclass[10pt,a4paper,reqno]{amsart}
\usepackage{amsthm}
\usepackage{amsmath}
\usepackage{amssymb}
\usepackage{graphicx,color}

\usepackage[font=small]{caption}
\usepackage{subcaption}

\usepackage{multirow}
\usepackage[shortlabels]{enumitem}

\makeatletter 
\@namedef{subjclassname@2020}{\textup{2020} Mathematics Subject Classification}
\makeatother

\makeatletter
\theoremstyle{plain}
\newtheorem{thm}{Theorem}
  \theoremstyle{definition}
  \newtheorem*{thm*}{Theorem}
  \newtheorem{defn}{Definition}
  \theoremstyle{remark}
  \newtheorem{rem}[thm]{Remark}
  \theoremstyle{plain}
  \newtheorem{prop}[thm]{Proposition}
  \theoremstyle{plain}
  \newtheorem{lem}[thm]{Lemma}
  \theoremstyle{plain}
  
 \theoremstyle{definition}
  
  \theoremstyle{remark}
  \newtheorem*{rem*}{Remark}

  \theoremstyle{definition}

\usepackage{amsfonts}
\usepackage{mathrsfs}

\newtheorem*{question*}{\it{QUESTION}}

\theoremstyle{plain}

\newcommand{\N}{\mathbb{N}}
\newcommand{\R}{{\mathbb{R}}}
\newcommand{\C}{{\mathbb{C}}}

\newcommand{\Z}{{\mathbb{Z}}}
\newcommand{\dd}{{\rm d}}

\renewcommand{\gg}{{\mathfrak{g}}}

\renewcommand{\Re}{\mathop\mathrm{Re}\nolimits}

\newcommand{\grad}{\mathop\nabla\nolimits}
\renewcommand{\div}{\mathop\mathrm{div}\nolimits}

\DeclareMathOperator{\shift}{S}
\renewcommand{\SS}{{\shift}}

\DeclareMathOperator{\midop}{M}
\newcommand{\M}{{\midop}}

\DeclareMathOperator{\identity}{I}
\newcommand{\I}{{\identity}}

\DeclareRobustCommand{\bigO}{
  \text{\usefont{OMS}{cmsy}{m}{n}O}%
}

\makeatother

\begin{document}

\title[]{Optimal discrete Hardy--Rellich--Birman inequalities}

\author{Franti\v sek \v Stampach}
\address[Franti{\v s}ek {\v S}tampach]{
	Department of Mathematics, Faculty of Nuclear Sciences and Physical Engineering, Czech Technical University in Prague, Trojanova~13, 12000 Praha~2, Czech Republic
	}	
\email{stampfra@cvut.cz}

\author{Jakub Waclawek}
\address[Jakub Waclawek]{
	Department of Mathematics, Faculty of Nuclear Sciences and Physical Engineering, Czech Technical University in Prague, Trojanova~13, 12000 Praha~2, Czech Republic
	}	
\email{waclajak@cvut.cz}



\date{\today}

\begin{abstract}
We prove sufficient conditions on a parameter sequence to determine optimal weights in inequalities for an integer power $\ell$ of the discrete Laplacian on the half-line. By a concrete choice of the parameter sequence, we obtain explicit optimal discrete Rellich ($\ell=2$) and Birman ($\ell\geq3$) weights. For $\ell=1$, we rediscover the optimal Hardy weight of Keller--Pinchover--Pogorzelski. For $\ell=2$, we improve upon the best known Rellich weights due to Gerhat--Krej{\v c}i{\v r}{\' i}k--{\v S}tampach and Huang--Ye. For $\ell\geq3$, our main result proves a conjecture by Gerhat--Krej{\v c}i{\v r}{\' i}k--{\v S}tampach and improves the discrete analogue of the classical Birman weight due to Huang--Ye to the optimal.
\end{abstract}

\maketitle

\section{Introduction}

A century passed since G.~H.~Hardy discovered his famous inequality
\begin{equation}
\sum_{n=1}^{\infty}|u_{n}-u_{n-1}|^{2}\geq\frac{1}{4}\sum_{n=1}^{\infty}\frac{|u_{n}|^{2}}{n^{2}},
\label{eq:hardy_class}
\end{equation}
which is true for any $u\in\ell^{2}(\N_{0})$ with $u_{0}=0$. In fact, many great mathematicians such as E. Landau, G. Pólya, I. Schur, and M. Riesz contributed to the early stage developments of Hardy inequalities; see~\cite{kuf-mal-per_06, per-sam_24} and references therein for historical account. Since then a tremendous number of variants and generalizations of the Hardy inequality has been studied and found applications in various areas such as probability, geometry, PDEs, spectral theory, or mathematical physics. The American Mathematical Society MathSciNet database involves 1317 articles, 8 books, and 7 theses with "Hardy" and "inequalities" in the title on April 25, 2024. Most of these works focus on Hardy inequalities in a continuous setting; the continuous analogue of~\eqref{eq:hardy_class} reads
\begin{equation}
 \int_{0}^{\infty}\left|u'(x)\right|^{2}\dd x\geq\frac{1}{4}\int_{0}^{\infty}\frac{|u(x)|^{2}}{x^{2}}\,\dd x
\label{eq:hardy_class_cont}
\end{equation}
and holds true for any function $u$ from the Sobolev space $H^{1}(0,\infty)$ with $u(0)=0$ (the form domain of the Dirichlet Laplacian in $L^{2}(0,\infty)$). The two inequalities can be shown to be equivalent, i.e. one can deduce~\eqref{eq:hardy_class_cont} from~\eqref{eq:hardy_class} and vice versa.

Approximately a century after the appearance of~\eqref{eq:hardy_class}, M.~Keller, Y.~Pinchover, and F.~Pogorzelski made an interesting observation that, although the constant $1/4$ in~\eqref{eq:hardy_class} is the best possible, the weight sequence $1/(4n^{2})$ can be still improved. These authors found in~\cite{kel-pin-pog_18a,kel-pin-pog_18b} the improved Hardy weight
\[
 \rho_{n}^{\rm{KPP}}:= 2-\sqrt{1-\frac{1}{n}}-\sqrt{1+\frac{1}{n}}>\frac{1}{4n^{2}},
\]
i.e. inequality~\eqref{eq:hardy_class} is still true when the weight $1/(4n^{2})$ on the right-hand side is replaced by $\rho_{n}^{\rm{KPP}}$; see also~\cite{kre-sta_22} for a simple proof and~\cite{fis-kel-pog} for an $\ell^{p}$-generalization. Moreover, the weight $\rho^{\rm{KPP}}$ was shown to be optimal in~\cite{kel-pin-pog_18b}, see also~\cite{ger-kre-sta_23}. The notion of \emph{optimality} is a rather strong property which was introduced in discrete setting of graphs in~\cite{kel-pin-pog_18b} adapted from earlier work~\cite{dev-fra-pin_14} on Hardy inequalities for PDEs (see Definition~\ref{def:optim} below). In particular, the optimality of $\rho^{\rm{KPP}}$ implies that the discrete Hardy inequality does not hold with any point-wise greater weight sequence, hence $\rho^{\rm{KKP}}$ cannot be improved any further in this sense. This fact is interesting since it contrasts with the continuous setting~\eqref{eq:hardy_class_cont}, where it is well known that the Hardy weight $1/(4x^{2})$ is critical, meaning that, if~\eqref{eq:hardy_class_cont} holds with $1/(4x^{2})$ replaced by a measurable function $\rho(x)\geq1/(4x^{2})$ for a.e.~$x>0$, then $\rho(x)=1/(4x^{2})$ for a.e.~$x>0$.

Hardy classical inequalities~\eqref{eq:hardy_class} or~\eqref{eq:hardy_class_cont} can be interpreted in the sense of quadratic forms in $\ell^{2}(\N_{0})$ or $L^{2}(0,\infty)$ as lower bounds for the discrete or continuous Dirichlet Laplacian on the half-line, $-\Delta\geq\rho$, where $\rho$ stands for the operator of multiplication by either the discrete or the continuous Hardy weight. In this article, we use the following definition of the \emph{discrete Laplacian},
\begin{equation}
 (\Delta u)_{n}:=u_{n-1}-2u_{n}+u_{n+1}
\label{eq:def_disc_lapl}
\end{equation}
acting on the space of complex sequences $u$ indexed by $\Z$, whose domain will be restricted further below. This definition of $\Delta$ differs by a sign from the definition of the combinatorial Laplacian used by other authors~\cite{huang-ye_24, kel-pin-pog_18b}. With our sign convention, $-\Delta$ is a nonnegative operator on spaces of square summable sequences and the Hardy inequality takes the form $-\Delta\geq\rho$ on respective spaces in both the discrete as well as the continuous setting.

Lower bounds for the second and higher integer powers of the (continuous) Dirichlet Laplacian on the half-line were studied by F.~Rellich and M.~{\v S}.~Birman. Rellich's inequality is the lower bound for the bi-Laplacian whose one-dimensional form reads
\[
 \int_{0}^{\infty}\left|u''(x)\right|^2\dd x\geq\frac{9}{16}\int_{0}^{\infty}\frac{|u(x)|^{2}}{x^{4}}\dd x,
\]
where $u\in H^{2}(0,\infty)$ with $u(0)=u'(0)=0$. Rellich's inequality was published posthumously in~\cite{rel_56}. Birman~\cite{bir_61} generalized the Hardy and Rellich inequalities by considering derivatives of order $\ell\in\N$ and obtained inequality
\begin{equation}
 \int_{0}^{\infty}|u^{(\ell)}(x)|^2\dd x\geq \frac{\left((2\ell)!\right)^{2}}{16^{\ell}\left(\ell!\right)^{2}} \int_{0}^{\infty}\frac{|u(x)|^{2}}{x^{2\ell}}\dd x
\label{eq:birman_ineq_cont}
\end{equation}
for $u\in H^{\ell}(0,\infty)$ satisfying $u(0)=\dots=u^{(\ell-1)}(0)=0$, with the best possible constant; see also Glazman's book~\cite[pp.~83--84]{gla_65} for a detailed proof and articles~\cite{ges-lit-mic_18, owe_99} for other proofs. Refinements, weighted variants and other generalizations of~\eqref{eq:birman_ineq_cont} appeared only recently~\cite{ges-etal_22a, ges-etal_22b, ges-etal_22c}.

\subsection{State of the art}

Being aware of the fact that the classical Hardy inequality~\eqref{eq:hardy_class} admits an improvement, it is reasonable to expect the same for the discrete analogue of the Rellich and Birman inequalities
\begin{equation}
\sum_{n=\lceil\ell/2\rceil}^{\infty}\left|(-\Delta)^{\ell/2}u_{n}\right|^{2}\geq\frac{\left((2\ell)!\right)^{2}}{16^{\ell}\left(\ell!\right)^{2}} \sum_{n=\ell}^{\infty}\frac{|u_{n}|^{2}}{n^{2\ell}}
\label{eq:birman_ineq_disc}
\end{equation}
for $\ell\geq2$, where $u\in\ell^{2}(\N_{0})$ with $u_{0}=\dots=u_{\ell-1}=0$, $\lceil x\rceil$ is the lowest integer greater or equal to a real number $x$, and 
\[
 (-\Delta)^{\ell/2}:= \begin{cases} (-\Delta)^{m} &\quad\mbox{ if } \ell=2m, \\ 
									\grad\circ(-\Delta)^{m} &\quad\mbox{ if } \ell=2m+1.
 \end{cases} 
\]
Here we adopt the notation for the \emph{discrete gradient} and \emph{divergence} acting on sequences $u$ indexed by $\Z$ by formulas
\begin{equation}
 (\grad u)_{n}:=u_{n}-u_{n-1} \quad\mbox{ and }\quad (\div u)_{n}:=u_{n+1}-u_{n}
\label{eq:def_grad_div}
\end{equation}
for all $n\in\Z$. By~\eqref{eq:def_disc_lapl}, we have the familiar equality for the Laplacian $\Delta=\div\circ\grad$ on the space of complex sequences indexed by $\Z$. We omit $\circ$ when composing difference operators as well as the brackets writing $\grad u_{n}$, $\div u_{n}$, $\Delta u_{n}$, etc. to simplify the notation below. The left-hand side of~\eqref{eq:birman_ineq_disc} coincides with the quadratic form of $(-\Delta)^{\ell}$, i.e. with $\langle u, (-\Delta)^{\ell}u\rangle$, where $\langle\cdot,\cdot\rangle$ is the Euclidean inner product in $\ell^{2}(\N_{0})$. Inequality~\eqref{eq:birman_ineq_disc} was proven in~\cite{ger-kre-sta_23} for $\ell=2$ and in~\cite{huang-ye_24} for $\ell\geq3$. The constant on the right-hand side in~\eqref{eq:birman_ineq_disc} is the best possible, which was shown in~\cite{huang-ye_24}, too.

First steps towards an improvement of~\eqref{eq:birman_ineq_disc} in the Rellich case $\ell=2$ were done in~\cite{ger-kre-sta_23} by proving that 
\[
\sum_{n=1}^{\infty}\left|\Delta u_{n}\right|^{2}\geq
\sum_{n=2}^{\infty}\rho_{n}^{\rm{GKS}}|u_{n}|^{2},
\]
with
\[
 \rho_{n}^{\rm{GKS}}:=\frac{\Delta^{2} n^{3/2}}{n^{3/2}}=6-4\left(1-\frac{1}{n}\right)^{\!3/2}-4\left(1-\frac{1}{n}\right)^{\!3/2}+\left(1-\frac{2}{n}\right)^{\!3/2}+\left(1+\frac{2}{n}\right)^{\!3/2}>\frac{9}{16n^{4}}
\]
for $n\geq2$. Moreover, the authors of~\cite{ger-kre-sta_23} conjectured that the discrete Birman inequality~\eqref{eq:birman_ineq_disc}, for all $\ell\geq3$, can be improved to 
\begin{equation}
\sum_{n=\lceil\ell/2\rceil}^{\infty}\left|(-\Delta)^{\ell/2}u_{n}\right|^{2}\geq\sum_{n=\ell}^{\infty}\frac{(-\Delta)^{\ell} n^{\ell-1/2}}{n^{\ell-1/2}}|u_{n}|^{2},
\label{eq:birman_ineq_improved_conj}
\end{equation}
and showed that this form would really improve upon~\eqref{eq:birman_ineq_disc}. This is true, indeed, nevertheless neither this form is optimal, which follows from Theorem~\ref{thm:5} below.

Shortly after~\cite{ger-kre-sta_23}, yet another discrete Rellich weight of a quite complicated form
\[
 \rho_{n}^{\rm{HY}}:=\frac{1}{4}\left(A_{n}+\sum_{k=2}^{\infty}\frac{(2k+1)^{2}r_{2k}^{(1)}}{n^{2k+2}}\right),
\]
with 
\[
 A_{n}:=\frac{1}{n^{2}}\left[1+\left(1-\frac{1}{n}\right)^{\!-2}-\left(1-\frac{1}{n}\right)^{\!-1/2}-\left(1+\frac{1}{n}\right)^{\!3/2} \right]
\]
and $r_{2k}^{(1)}$ the positive coefficients defined by~\eqref{eq:r_k^1} below, has been found in~\cite{huang-ye_24}. The Rellich weight $\rho^{\rm{HY}}$ improves upon $\rho^{\rm{GKS}}$ at least asymptotically for large index as one sees from the comparison of the second terms in their asymptotic expansions  
\begin{align}
\rho_{n}^{\rm{YH}}&=\frac{9}{16n^{4}}+\frac{15}{16n^{5}}+\frac{213}{128n^{6}}+\bigO\left(\frac{1}{n^{7}}\right),\label{eq:rho_HY_asympt} \\
\rho_{n}^{\rm{GKS}}&=\frac{9}{16n^{4}}+\frac{105}{128n^{6}}+\frac{6237}{4096n^{8}}+\bigO\left(\frac{1}{n^{10}}\right) \nonumber
\end{align}
for $n\to\infty$. Yet an optimal weight even for the discrete Rellich inequality has remained unknown until now. The aim of the present article is to construct optimal discrete Rellich and general Birman weights and so establish an optimal improvement of~\eqref{eq:birman_ineq_disc} for general power $\ell\geq2$.

Further discussion on other recent and closely related results on lower bounds for powers of the discrete Laplacian is postponed to Subsection~\ref{subsec:matr_form} below.

\subsection{Organization of the paper}

In Subsection~\ref{subsec:main}, we formulate our main results as Theorems~\ref{thm:1}--\ref{thm:5} that are proved in Subsections~\ref{subsec:thm1_proof}--\ref{subsec:thm5_proof}, respectively. Connections to Toeplitz matrices and other works on powers of the discrete Laplacian are discussed in Subsection~\ref{subsec:matr_form}. In the final Section~\ref{sec:gen_weights}, we give a few secondary results on more general parameter families of Hardy--Rellich--Birman weights addressing their non-uniqueness and optimality. The paper is concluded by an appendix where two auxiliary statements are proven.

\subsection{Main results}\label{subsec:main}

We formulate our main results as five theorems whose proofs are gradually worked out in Section~\ref{sec:proof} below. Theorems~\ref{thm:1}--\ref{thm:3} give sufficient conditions that, when imposed to a \emph{parameter sequence} $\mathfrak{g}$, give rise to an optimal discrete Hardy--Rellich--Birman weight. By using the fraktur font for the sequence $\gg$, we want to emphasize the distinguished role of $\gg$ as the only parameter on which the constructed weights depend. With an explicit choice of $\mathfrak{g}=\mathfrak{g}^{(\ell)}$ depending on an integer~$\ell$ -- the power of $\Delta$ -- we demonstrate in Theorem~\ref{thm:4} that Theorems~\ref{thm:1}--\ref{thm:3} apply to~$\mathfrak{g}^{(\ell)}$ and analyze the resulting optimal discrete Hardy--Rellich--Birman weights in greater detail in Theorem~\ref{thm:5}.

It turns out to be advantageous to work with complex sequences indexed by $\Z$ with zero entries up to a certain positive index. For this reason, we introduce the following subspaces of the space of all complex sequences,
\[
 H^{\ell}:=\{u:\Z\to\C \mid u_{n}=0 \mbox{ for all } n<\ell\}, \quad 
 \mathcal{H}^{\ell}:= H^{\ell} \cap \ell^{2}(\Z),
\]
and
\[
 \mathcal{H}_{0}^{\ell}:=\{u\in H^{\ell} \mid u \mbox{ compactly supported}\}.
\]
Clearly, the subspace $\mathcal{H}^{\ell}$ can be naturally identified with $\ell^{2}(\N)$ for any $\ell\in\N$.
Recall the definition of \emph{optimality} adopted from~\cite{kel-pin-pog_18b} in a slightly modified form.

\begin{defn}\label{def:optim}
Let $\ell\in\N$. A nonnegative sequence $\{\rho_{n}\}_{n=\ell}^{\infty}$ is said to be a \emph{discrete Hardy} ($\ell=1$), \emph{Rellich} ($\ell=2$), or \emph{Birman} ($\ell\geq3$) \emph{weight} if and only if we have the inequality 
\begin{equation}
\sum_{n=\lceil\ell/2\rceil}^{\infty}\left|(-\Delta)^{\ell/2}u_{n}\right|^{2}\geq\sum_{n=\ell}^{\infty}\rho_{n}|u_{n}|^{2}
\label{eq:hrb_ineq_disc}
\end{equation}
for all $u\in\mathcal{H}_{0}^{\ell}$. In addition, the weight $\rho$ is said to be \emph{optimal} if the following three properties hold:
\begin{enumerate}[i)]
\item (\emph{Criticality}) If~\eqref{eq:hrb_ineq_disc} holds with $\rho$ replaced by another weight~$\tilde{\rho}$ such that $\tilde{\rho}_{n}\geq\rho_{n}$ for all $n\geq\ell$, then $\rho_{n}=\tilde{\rho}_{n}$ for all $n\geq\ell$.
\item (\emph{Non-attainability}) If~\eqref{eq:hrb_ineq_disc} holds as equality for a sequence $u\in H^{\ell}$ such that $\sqrt{\rho}u$ is square summable (i.e. the right-hand side of~\eqref{eq:hrb_ineq_disc} is finite), then $u\equiv0$.
\item (\emph{Optimality near infinity}) For every $M\geq\ell$ and $\varepsilon>0$, there exists $u\in\mathcal{H}_{0}^{M}$ such that
\begin{equation}
\sum_{n=\lceil\ell/2\rceil}^{\infty}\left|(-\Delta)^{\ell/2}u_{n}\right|^{2}<(1+\varepsilon)\sum_{n=\ell}^{\infty}\rho_{n}|u_{n}|^{2}.
\label{eq:def_opt_near_inf}
\end{equation}
\end{enumerate}
\end{defn}

The criticality means that $\rho$ in~\eqref{eq:hrb_ineq_disc} cannot be replaced by any point-wise greater weight. The non-attainability implies that for all non-trivial $u$ from the $\rho$-weighted $\ell^{2}$-space, inequality~\eqref{eq:hrb_ineq_disc} is strict.
Non-attainability together with criticality is called \emph{null-criticality} in~\cite{kel-pin-pog_18b}. Lastly, optimality near infinity implies that the constant $1$ appearing by the weight $\rho$ on the right-hand side of~\eqref{eq:hrb_ineq_disc} cannot be replaced by anything greater even if the space of sequences $u$ is restricted to compactly supported sequences that vanishes at an arbitrary finite number of first indices.

Inequality~\eqref{eq:hrb_ineq_disc} can be equivalently formulated in the sense of quadratic forms on $\mathcal{H}^{\ell}_{0}$ as $(-\Delta)^{\ell}\geq\rho$, where we again identify $\rho$ with the corresponding multiplication operator. In fact, inequality $(-\Delta)^{\ell}\geq\rho$ extends to $\mathcal{H}^{\ell}$ since $(-\Delta)^{\ell}$ is a bounded operator on $\mathcal{H}^{\ell}$. Actually~\eqref{eq:hrb_ineq_disc} can be shown to hold for any $u\in H^{\ell}$ if $\rho$ is positive, as one can verify by adapting arguments of Lemma~\ref{lem:extend} below.

Our method relies on an \emph{iterative use} of an identity for the quadratic form of the discrete Laplacian with an additional weight $V$ that implies a corresponding Hardy inequality and also identifies the \emph{remainder} in the inequality. This is an initial step for our deduction formulated below as Theorem~\ref{thm:0}. The key idea for the construction of optimal discrete Hardy--Rellich--Birman weights is in a convenient choice of the weight $V$ in terms of $\mathfrak{g}$ in each step of the iteration.

\setcounter{thm}{-1}
\begin{thm}\label{thm:0}
Let $V:\Z\to\C$ and $\gg\in H^{1}$ be such that $\gg_{n}>0$ for all $n\geq1$. Then for all $u\in\mathcal{H}^{1}_{0}$, we have the identity
\begin{equation}
		\sum_{n=1}^{\infty} V_{n} |\grad u_{n}|^{2} + \sum_{n=1}^{\infty} \frac{\div(V\grad \gg)_{n}}{\gg_n}|u_{n}|^{2} = \sum_{n=1}^{\infty} V_{n+1} \left|\sqrt{\frac{\gg_{n}}{\gg_{n+1}}}u_{n+1}-\sqrt{\frac{\gg_{n+1}}{\gg_{n}}}u_{n}\right|^2 .
\label{eq:hardy_id_init}
\end{equation}
\end{thm}

In a form related to the optimal Hardy inequality, identity~\eqref{eq:hardy_id_init} appeared in~\cite{kre-sta_22}, its non-weighted form, i.e. with $V\equiv 1$, was given in~\cite{kre-lap-sta_jlms22}, and proved in full generality in~\cite{huang-ye_24}. For the reader's convenience, we prove Theorem~\ref{thm:0} in Subsection~\ref{subsec:thm0_proof}.

Our first main result is a similar identity for the quadratic form of $(-\Delta)^{\ell}$ with arbitrary $\ell\geq1$, which singles out a term with the Hardy--Rellich--Birman weight and identifies the remainder explicitly in terms of $\gg$. For the iterative application of Theorem~\ref{thm:0}, it is necessary to impose certain positivity assumptions on the parameter sequence $\gg$, see \eqref{eq:assum_A1} below. In fact, in Theorems~\ref{thm:1}--\ref{thm:3}, the assumptions imposed on~$\gg$ will be gradually strengthened. These necessary conditions on~$\gg$ deserve to be emphasized and therefore a special numbering is used for them.  By convention, difference operators to the power $0$ such as $\div^{0}$ and $(-\Delta)^{0}$ are to be understood as the identity operator.

\begin{thm}\label{thm:1}
Let $\ell\in\N$. Suppose
\begin{equation}
\gg\in H^{\ell} \mbox{ with } \div^{k}\gg_{n}>0 \mbox{ for all } n\geq\ell-k \mbox{ and } 0\leq k<\ell.
\tag{\bf{A1}}
\label{eq:assum_A1}
\end{equation}
Then for all $u\in\mathcal{H}_{0}^{\ell}$, we have the identity
\begin{equation}
\sum_{n=\lceil\ell/2\rceil}^{\infty}\left|(-\Delta)^{\ell/2}u_{n}\right|^{2}=\sum_{n=\ell}^{\infty}\frac{(-\Delta)^{\ell}\gg_{n}}{\gg_{n}}\,|u_{n}|^{2}+\sum_{k=0}^{\ell-1}\mathcal{R}_{k}^{(\ell)}(\gg;u),
\label{eq:hrb_id_init}
\end{equation}
where 
\begin{align}
\mathcal{R}_{k}^{(\ell)}(\gg;u)&:=\sum_{n=\ell-k}^{\infty}\frac{(-\Delta)^{\ell-1-k}\div^{k+1}\gg_{n}}{\div^{k+1}\gg_{n}}\left|\sqrt{\frac{\div^{k}\gg_{n}}{\div^{k}\gg_{n+1}}}\,\div^{k}u_{n+1}-\sqrt{\frac{\div^{k}\gg_{n+1}}{\div^{k}\gg_{n}}}\,\div^{k}u_{n}\right|^{2}\!. \nonumber\\
\label{eq:def_remainder_R}
\end{align}
(For $k=\ell-1$, the coefficient in front of the absolute value in~\eqref{eq:def_remainder_R} is to be interpreted as $1$.)
\end{thm}

Next, we strengthen positivity assumption~\eqref{eq:assum_A1} imposed on the parameter sequence $\gg$  to ensure non-negativity of remainder terms~\eqref{eq:def_remainder_R} obtaining an abstract discrete Hardy--Rellich--Birman inequality.

\begin{thm}\label{thm:2}
Let $\ell\in\N$. Suppose~\eqref{eq:assum_A1} and
\begin{equation}
(-\Delta)^{\ell-k}\div^{k}\gg_{n}\geq0 \mbox{ for all } n\geq\ell+1-k \mbox{ and } 1\leq k<\ell.
\tag{\bf{A2}}
\label{eq:assum_A2}
\end{equation}
Then for all $u\in\mathcal{H}^{\ell}_{0}$, we have the inequality 
\begin{equation}
\sum_{n=\lceil\ell/2\rceil}^{\infty}\left|(-\Delta)^{\ell/2}u_{n}\right|^{2}\geq\sum_{n=\ell}^{\infty}\rho_{n}(\gg)|u_{n}|^{2},
\label{eq:hrb_ineq_thm2}
\end{equation}
where $\rho(\gg):=(-\Delta)^{\ell}\gg/\gg$. If in addition,
\begin{equation}
(-\Delta)^{\ell}\gg_{n}\geq0 \mbox{ for all } n\geq\ell,
\tag{\bf{A2'}}
\label{eq:assum_A2'}
\end{equation}
then $\rho(\gg)\geq0$, i.e. $\rho(\gg)$ is a discrete Hardy--Rellich--Birman weight.
\end{thm}

By imposing an additional requirement on the asymptotic behavior of $\gg_{n}$ for $n$ large and strict positivity in assumptions \eqref{eq:assum_A2'} and \eqref{eq:assum_A2} for $k=1$ (if $\ell\geq2$), we obtain sufficient conditions for the optimality of the discrete Hardy--Rellich--Birman weights of Theorem~\ref{thm:2} in the next statement.

\begin{thm}\label{thm:3}
Let $\ell\in\N$. Suppose~\eqref{eq:assum_A1}, \eqref{eq:assum_A2}, \eqref{eq:assum_A2'}, and $\gg$ to admit the asymptotic expansion
\begin{equation}
\gg_{n}=\sum_{j=0}^{2\ell}\alpha_{j}n^{\ell-j-s}+\bigO\left(n^{-\ell-1-s}\right), \quad\mbox{ as } n\to\infty,
\label{eq:assum_asympt_g}
\end{equation}
for some $\alpha_{j}\in\R$ with $\alpha_{0}\neq0$ and $s\in(0,1)$. Let $\rho(\gg):=(-\Delta)^{\ell}\gg/\gg$ denote the weight from~\eqref{eq:hrb_ineq_thm2}. Then we have:
\begin{enumerate}[{\upshape i)}]
\item If the expansion
\begin{equation}
\mbox{\eqref{eq:assum_asympt_g} holds with } s\geq 1/2,
\tag{\bf{A3}}
\label{eq:assum_A3}
\end{equation}
then $\rho(\gg)$ is critical.
\item If the expansion
\begin{equation}
\mbox{\eqref{eq:assum_asympt_g} holds with } s=1/2,
\tag{\bf{A3'}}
\label{eq:assum_A3'}
\end{equation}
then $\rho(\gg)$ is optimal near infinity.
\item If the expansion
\begin{equation}\tag{\bf{A3''}}\label{eq:assum_A3''}
\begin{aligned}
&\hskip112pt \mbox{\eqref{eq:assum_asympt_g} holds with } s\leq1/2,\; \\
&(-\Delta)^{\ell}\gg_{n}>0, 
\mbox{ and, if }\; \ell\geq2, \mbox{ also } (-\Delta)^{\ell-1}\div\gg_{n}>0,  \mbox{ for all } n\geq\ell,
\end{aligned}
\end{equation}
then $\rho(\gg)$ is non-attainable.
\end{enumerate}
In particular, if $\gg$ fulfills assumptions~\eqref{eq:assum_A1}, \eqref{eq:assum_A2}, and~\eqref{eq:assum_A3''} with $s=1/2$, then $\rho(\gg)$ is a~strictly positive optimal discrete Hardy--Rellich--Birman weight.
\end{thm}

Next, we turn to a concrete choice of the parameter sequence. Given $\ell\in\N$, we put
\begin{equation}
\mathfrak{g}^{(\ell)}_{n}:=\sqrt{n}\prod_{j=1}^{\ell-1}(n-j)
\label{eq:def_gg_l}
\end{equation}
for $n\geq0$ and $\gg_{n}^{(\ell)}:=0$ for $n<0$. It turns out that $\gg^{(\ell)}$ satisfies each of the assumptions of Theorems \ref{thm:1}--\ref{thm:3}. Consequently, $\gg^{(\ell)}$ gives rise to a concrete  optimal strictly positive discrete Hardy--Rellich--Birman weight. Further concrete (but more complicated) examples are discussed in Section~\ref{sec:gen_weights}.

\begin{thm}\label{thm:4}
For any $\ell\in\N$, the weight $\rho^{(\ell)}$ given by
\begin{equation}
\rho_{n}^{(\ell)}:=\frac{(-\Delta)^{\ell}\gg_{n}^{(\ell)}}{\gg_{n}^{(\ell)}}
\label{eq:def_rho_l}
\end{equation}
for $n\geq\ell$, is the optimal strictly positive discrete Hardy--Rellich--Birman weight.
\end{thm}

Our final theorem summarizes properties of the optimal weight $\rho^{(\ell)}$ in greater detail. A~remarkable property is that, for $n\geq\ell\geq2$, $\rho^{(\ell)}_{n}$ has a convergent series representation in negative powers of $n$ with all coefficients \emph{positive}. Consequently, using more terms of the truncated series representation always produces a tighter inequality. The leading term yields the classical discrete Birman weight of~\eqref{eq:birman_ineq_disc}.

To formulate our final theorem (and subsequent remarks), we need to introduce several combinatorial numbers. First, the \emph{binomial number} and the \emph{Pochhamer symbol} are defined by the standard formulas
\[
 \binom{\nu}{n}:=\frac{\nu(\nu-1)\dots(\nu-n+1)}{n!} \quad\mbox{ and }\quad  (\nu)_{n}:=\nu(\nu+1)\dots(\nu+n-1)
\]
for any $\nu\in\R$ and $n\in\N_{0}$. Next, for $n\in\N$ and $0\leq k< n$, we denote 
\begin{equation}
 s(n,k):=(-1)^{n+k}\sum_{1\leq i_{1}<\dots<i_{n-k}<n}i_{1}i_{2}\dots i_{n-k}
\label{eq:def_stirling_1st}
\end{equation}
and
\begin{equation}
 S(n,k):=\sum_{\substack{j_{1},\dots,j_{k}\geq0 \\ j_{1}+\dots+j_{k}=n-k}}1^{j_{1}}2^{j_{2}}\dots k^{j_{k}}
\label{eq:def_stirling_2nd}
\end{equation}
the \emph{Stirling numbers of the first and second kind}, respectively, see~\cite[\S~26.8]{dlmf}. For $n=k\geq0$, $s(n,n)=S(n,n):=1$. By convention, we also put $s(n,k):=0$ for $k<0$.
Finally, we will make use of numbers
\begin{equation}
 X_{m}^{(\ell)}:=\sum_{j=-\ell}^{\ell}\binom{2\ell}{\ell+j}(-1)^{j}j^{m}
\label{eq:def_X_ml}
\end{equation}
for $m,\ell\in\N$; $X_{0}^{(\ell)}:=0$ for all $\ell\in\N$. It is obvious that $X_{m}^{(\ell)}=0$ if $m$ is odd.
Moreover, we know from \cite[Sec.~4]{ger-kre-sta_23} that 
\begin{equation}
X_{m}^{(\ell)}=0, \quad \forall m<2\ell,
\label{eq:X_vanish_id}
\end{equation}
and that for the remaining values we have
\begin{equation}
X_{2\ell}^{(\ell)}=(-1)^{\ell}(2\ell)! \quad\mbox{ and }\quad X_{2\ell+2r}^{(\ell)}
=(-1)^{\ell}(2\ell)!\sum_{1\leq k_{1}\leq\dots \leq k_{r}\leq\ell}(k_{1}k_{2}\dots k_{r})^{2}
\label{eq:X_id_pos}
\end{equation}
for $r\in\N$, see~\cite[Lem.~4.1]{ger-kre-sta_23}. Expression~\eqref{eq:X_id_pos} reveals the nontrivial fact that $(-1)^{\ell}X_{2\ell+2r}^{(\ell)}$ is a~positive integer for all $r\in\N_{0}$.

\begin{thm}\label{thm:5}
Let $\ell\in\N$ and $\rho^{(\ell)}$ be defined by formulas~\eqref{eq:def_rho_l} and~\eqref{eq:def_gg_l}.
\begin{enumerate}[{\upshape i)}]
\item Weight sequence $\rho^{(\ell)}$ admits the convergent series expansion 
\begin{equation}
 \rho_{n}^{(\ell)}=\frac{n^{\ell-1}}{(n-1)\dots(n-\ell+1)}\sum_{k=2\ell}^{\infty}\frac{r_{k}^{(\ell)}}{n^{k}}
\label{eq:rho_expansion}
\end{equation}
for all $n\geq\ell$, where 
\begin{equation}
 r_{k}^{(\ell)}=\sum_{m=2\ell}^{k}\binom{\ell+m-k-1/2}{m}s(\ell,\ell+m-k)X_{m}^{(\ell)}.
\label{eq:coeff_r_kl}
\end{equation}
In particular, we have
\[
 r_{2\ell}^{(\ell)}=\left(\frac{1}{2}\right)_{\ell}^{2} 
 \quad\mbox{ and }\quad
 r_{2\ell+1}^{(\ell)}=\frac{\ell(\ell-1)(2l+1)}{2(2l-1)}\left(\frac{1}{2}\right)_{\ell}^{2}.
\]
\item For all $\ell\geq2$ and $k\geq2\ell$, $r_{k}^{(\ell)}>0$.\\ (If $\ell=1$, then $r_{2k+1}^{(1)}=0$  and $r_{2k}^{(1)}>0$ for all $k\geq1$.)
\item For all $n\geq\ell\geq2$, we have 
\[
\rho_{n}^{(\ell)}>\frac{(-\Delta)^{\ell} n^{\ell-1/2}}{n^{\ell-1/2}}>\left(\frac{1}{2}\right)_{\ell}^{2}\frac{1}{n^{2\ell}}.
\]
\end{enumerate}
\end{thm}

\begin{rem} We complement Theorem~\ref{thm:5} by several remarks.
\begin{enumerate}[a)]
\item Since $s(\ell,j)=0$ whenever $j\leq0$ and $X_{m}^{(\ell)}=0$ whenever $m$ is odd, we may restrict the range of the summation index $m$ in formula~\eqref{eq:coeff_r_kl} even further. The formula with each summand positive reads
\[
 r_{k}^{(\ell)}=\sum_{\substack{m=\max(2\ell,k-\ell+1) \\ m=0 \hskip-5pt \mod 2}}^{k}\binom{\ell+m-k-1/2}{m}s(\ell,\ell+m-k)X_{m}^{(\ell)}.
\]
\item We conjecture that $4^{k-\ell}r_{k}^{(\ell)}\in\N$ for all $\ell\geq2$ and $k\geq2\ell$. These integer sequences are not listed in \texttt{oeis.org} (by April 25, 2024).
\item By using only the terms in~\eqref{eq:rho_expansion} with $k=2\ell,2\ell+1,\dots,3\ell-1$ and a little algebra, we obtain the inequality
\[
\rho_{n}^{(\ell)}>\frac{n^{\ell-1}}{16^{\ell}(n-1)\dots(n-\ell+1)}\sum_{j=1}^{\ell}\frac{(2j)!(4\ell-2j)!}{j!(2\ell-j)!}\frac{|s(\ell,j)|}{n^{3\ell-j}}
\]
for all $n\geq\ell\geq1$.
\item If needed, one can easily expand also the prefactor in front of the sum from~\eqref{eq:rho_expansion} in terms of negative powers of $n$ and Stirling numbers of the second kind defined by~\eqref{eq:def_stirling_2nd}. Namely, we have
\[
 \frac{n^{\ell-1}}{(n-1)\dots(n-\ell+1)}=\sum_{j=0}^{\infty}\frac{S(j+\ell-1,\ell-1)}{n^{j}}
\]
for all $n\geq\ell\geq1$, see~\cite[Eq.~(26.8.11)]{dlmf}, and the complete expansion of $\rho_{n}^{(\ell)}$ then reads
\[
 \rho_{n}^{(\ell)}=\sum_{m=2\ell}^{\infty}\left(\sum_{k=2\ell}^{m}S(m-k+\ell-1,\ell-1)r_{k}^{(\ell)}\right)\frac{1}{n^{m}}.
\]
For $\ell\geq2$, every coefficient of this expansion is positive as one readily deduces from claim~(ii) of Theorem~\ref{thm:5} and~\eqref{eq:def_stirling_2nd}.
\item Expansions in first few terms for $\ell=2,3,4,5$, as $n\to\infty$, are listed below.
\begin{align*}
\rho_{n}^{(2)}&=\frac{9}{16n^{4}}+\frac{3}{2n^{5}}+\frac{297}{128n^{6}} + \bigO\left(\frac{1}{n^{7}}\right),\\
\rho_{n}^{(3)}&=\frac{225}{64n^{6}}+\frac{405}{16n^{7}}+\frac{114975}{1024n^{8}} + \bigO\left(\frac{1}{n^{9}}\right),\\
\rho_{n}^{(4)}&=\frac{11025}{256n^{8}}+\frac{4725}{8n^{9}}+\frac{4879665}{1024n^{10}} + \bigO\left(\frac{1}{n^{11}}\right),\\
\rho_{n}^{(5)}&=\frac{893025}{1024n^{10}}+\frac{2480625}{128n^{11}}+\frac{4023077625}{16384n^{12}} + \bigO\left(\frac{1}{n^{13}}\right).
\end{align*}
In~\cite[Sec.~7.2]{huang-ye_24}, the authors ask whether the constant $15/16$ that appears by the second term in the expansion of the Rellich weight~\eqref{eq:rho_HY_asympt} is sharp. The above expansion of $\rho_{n}^{(2)}$ shows that it is not the case.
\item The first inequality of claim (iii) in Theorem~\ref{thm:5} shows that $\rho^{(\ell)}$ improves upon the Birman weights suggested by Gerhat--Krej{\v c}i{\v r}{\' i}k--{\v S}tampach in~\cite[Sec.~4]{ger-kre-sta_23} and proves the conjecture~\eqref{eq:birman_ineq_improved_conj} formulated therein in the affirmative.
\item Noticing that
\[
 \left(\frac{1}{2}\right)_{\ell}^{2}=\frac{\left((2\ell)!\right)^{2}}{16^{\ell}\left(\ell!\right)^{2}},
\]
the second inequality of claim (iii) in Theorem~\ref{thm:5} shows that $\rho^{(\ell)}$ improves upon the discrete analogue to the classical Birman weights, see~\eqref{eq:birman_ineq_disc}.
\end{enumerate}
\end{rem}

\subsection{Matrix formulation and connections to other works}\label{subsec:matr_form}

By polarization, Theorem~\ref{thm:1} and~\ref{thm:2} yield an equality between sesquilinear forms on $\mathcal{H}_{0}^{\ell}$. Namely,
\begin{equation}
 \langle u, (-\Delta)^{\ell} v\rangle =\langle u,\rho(\gg) v\rangle + \sum_{k=0}^{\ell-1}\langle R_{k}^{(\ell)}(\gg)u,R_{k}^{(\ell)}(\gg)v\rangle
\label{eq:forms_id}
\end{equation}
for all $u,v\in\mathcal{H}_{0}^{\ell}$, where $\rho(\gg)=(-\Delta)^{\ell}\gg/\gg$,
\begin{align}
R_{k}^{(\ell)}(\gg)u_{n}&:=\sqrt{\frac{(-\Delta)^{\ell-1-k}\div^{k+1}\gg_{n}}{\div^{k+1}\gg_{n}}}\left(\sqrt{\frac{\div^{k}\gg_{n}}{\div^{k}\gg_{n+1}}}\,\div^{k}u_{n+1}-\sqrt{\frac{\div^{k}\gg_{n+1}}{\div^{k}\gg_{n}}}\,\div^{k}u_{n}\right)\!, \nonumber \\
\label{eq:def_Ru_critic_inproof}
\end{align}
if $n\geq\ell-k$, and $R_{k}^{(\ell)}(\gg)u_{n}:=0$, if $n<\ell-k$.

By taking $u=\delta_{n}$ and $v=\delta_{m}$ for $m,n\geq\ell$, where $\{\delta_{n}\mid n\in\Z\}$ denotes the standard basis of $\ell^{2}(\Z)$, equality \eqref{eq:forms_id} yields an algebraic identity on the level of semi-infinite matrices. Namely, restricting indices of the respective matrices to $\ell, \ell+1, \dots$, we obtain the identity
\begin{equation}
 (-\Delta)^{\ell}=\rho(\gg)+\sum_{k=0}^{\ell-1}\left(\tilde{R}_{k}^{(\ell)}(\gg)\right)^{*}\tilde{R}_{k}^{(\ell)}(\gg)
\label{eq:matrix_id}
\end{equation}
between semi-infinite matrices, where $\tilde{R}_{k}^{(\ell)}(\gg):=\SS^{-k}R_{k}^{(\ell)}(\gg)$ and $\SS^{-k}$ acts as the backward shift of the index by $k$, see~\eqref{eq:def_S} below. The shift $\SS^{-k}$ is present as a consequence of the range of summation index $n$ in~\eqref{eq:def_remainder_R} starting from $\ell-k$.
If $\ell=1$, such factorization has been used to provide an alternative proof of the optimal discrete Hardy inequality in~\cite{ger-kre-sta_23}. For $\ell=2$, the authors of~\cite{ger-kre-sta_23} factorized matrix $(-\Delta)^{2}-\rho(\gg)$, with $\gg_{n}=n^{3/2}$, into a single remainder matrix of form $X^{*}X$, where $X$ is a tridiagonal matrix. As the remainder was sought in terms of a single matrix rather than two matrices, its entries could not be found explicitly. The idea was to decompose the pentadiagonal matrix $(-\Delta)^{2}-\rho(\gg)$ into a product of a tridiagonal matrix and its adjoint reducing the order of the corresponding difference operators. 
A similar idea of factorization has been applied also in the continuous setting~\cite{ges-lit_18}.
  
The idea behind~\eqref{eq:matrix_id} is similar, however, its main novelty is that the order is reduced successively giving rise to more remainder terms on the right, but expressed fully explicitly in terms of the parameter sequence $\gg$. In addition, the diagonal term $\rho(\gg)$, i.e. the actual Hardy-Rellich-Birman weight, is identified in terms of $\gg$, too. Such an explicit description was essential for the discovery of the concrete Hardy--Rellich--Birman weights of Theorem~\ref{thm:4} and proof of their optimality.

The matrix identity~\eqref{eq:matrix_id} can be viewed as a factorization of particular banded Toeplitz matrices. Indeed, the non-vanishing matrix elements of $(-\Delta)^{\ell}$ are
\[
 (-\Delta)^{\ell}_{m,n}\equiv\langle\delta_{m},(-\Delta)^{\ell}\delta_{n}\rangle=(-1)^{n-m}\binom{2\ell}{\ell+n-m}
\]
for $m,n\geq\ell$ with $|n-m|\leq\ell$, i.e. the matrix representation of $(-\Delta)^{\ell}$ with respect to $\{\delta_{n} \mid n\geq\ell\}$ is a semi-infinite Hermitian banded Toeplitz matrix with diagonals given by the binomial coefficients. On the other hand, by inspection of matrix entries of remainder matrices~\eqref{eq:def_Ru_critic_inproof}, we observe that $\tilde{R}_{k}^{(\ell)}(\gg)$ are semi-infinite $(k+2)$-diagonal lower Hessenberg matrices, i.e. the $(m,n)$th entry of $\tilde{R}_{k}^{(\ell)}(\gg)$ vanishes if $n-m>1$ or $m-n>k$. Regardless of the application in discrete Hardy--Rellich--Birman inequalities, the non-trivial factorization  identity~\eqref{eq:matrix_id} can be of independent interest.

Further, we point out differences with a recent study~\cite{ger-kre-sta_arxiv24}, where criticality of a general positive power of the discrete Laplacian on the half-line has been analyzed. The definition of an operator assigned to a power of the discrete Laplacian studied in~\cite{ger-kre-sta_arxiv24} differs from the one examined in the recent paper. The authors of~\cite{ger-kre-sta_arxiv24} considered operator $T:=-\Delta|_{\ell^{2}(\N)}$ and defined its positive power $T^{\alpha}$, $\alpha>0$, by the standard functional calculus using the spectral resolution of $T$. While $T$ coincides with $-\Delta|\mathcal{H}^{1}$ after an obvious identification of spaces $\ell^{2}(\N)$ and $\mathcal{H}^{1}$, their integer powers differ by a finite rank operator. This can be readily seen from their matrix representations. For example, for $\ell=3$, operators $(-\Delta)^{3}$ and $T^{3}$ are determined by the semi-infinite matrices
\[
 \begin{pmatrix}
  20 & -15 & 6 & -1 \\
  -15 & 20 & -15 & 6 & -1 \\
  6 & -15 & 20 & -15 & 6 & -1 \\
  -1 & 6 & -15 & 20 & -15 & 6 & -1 \\
  & -1 & 6 & -15 & 20 & -15 & 6 & -1 \\
  & & \ddots & \ddots & \ddots & \ddots & \ddots & \ddots & \ddots &
 \end{pmatrix}
\]
and
\[
 \begin{pmatrix}
  14 & -14 & 6 & -1 \\
  -14 & 20 & -15 & 6 & -1 \\
  6 & -15 & 20 & -15 & 6 & -1 \\
  -1 & 6 & -15 & 20 & -15 & 6 & -1 \\
  & -1 & 6 & -15 & 20 & -15 & 6 & -1 \\
  & & \ddots & \ddots & \ddots & \ddots & \ddots & \ddots & \ddots &
 \end{pmatrix},
\]
respectively. Notice the matrix of $T^{3}$ is not Toeplitz and differs from the matrix of $(-\Delta)^{3}$ by the upper-left $2\times 2$ matrix. In general, matrices of $T^{\ell}$ and $(-\Delta)^{\ell}$ differ by an upper-left $(\ell-1)\times(\ell-1)$ matrix. In fact, the matrix of $(-\Delta)^{\ell}$ is a submatrix  of $T^{\ell}$ after removing the first $\ell-1$ rows and columns.

Both $(-\Delta)^{\ell}$ and $T^{\ell}$ determine nonnegative operators on $\ell^{2}(\N)$ with the same spectrum filling the interval $[0,4^{\ell}]$. It is proven in~\cite{ger-kre-sta_arxiv24} that $T^{\alpha}$ is critical if and only if $\alpha\geq3/2$ meaning that, if $T^{\alpha}\geq\rho\geq0$ with $\alpha\geq3/2$, then $\rho$ must be trivial. Hence nontrivial Hardy-like inequalities exist for $T^{\alpha}$ only if $\alpha\in(0,3/2)$, and some non-trivial weights (although not optimal) were found in~\cite{ger-kre-sta_arxiv24}. Clearly, this contrasts the situation with $(-\Delta)^{\ell}$ considered here since $(-\Delta)^{\ell}$ is subcritical on $\mathcal{H}^{\ell}$ for every $\ell\in\N$ in a complete analogy to the case of the Dirichlet Laplacian on the half-line.

A general positive power of $-\Delta$ could be also considered in the present meaning. The respective operator can be defined as the restriction $(-\Delta)^{\alpha}|_{\ell^{2}(\N)}$, where $(-\Delta)^{\alpha}$ acts on the full-line space $\ell^{2}(\Z)$ and is defined by the usual functional calculus. The operator $(-\Delta)^{\alpha}|_{\ell^{2}(\N)}$ is still Toeplitz but not banded for general $\alpha>0$. Finding optimal weights $\rho=\rho(\alpha)\geq0$ such that $(-\Delta)^{\alpha}|_{\ell^{2}(\N)}\geq\rho(\alpha)$ for non-integral $\alpha>0$ remains an interesting open problem for future research.
Operator $(-\Delta)^{\alpha}$, with $\alpha>0$, considered on the full-line space $\ell^{2}(\Z)$ is known to be critical if and only if $\alpha\geq1/2$, and optimal Hardy-like weights for $(-\Delta)^{\alpha}$ are known explicitly for all $\alpha\in(0,1/2)$, see~\cite{cia-ron_18,kel-nie_23}.

\section{Proofs}\label{sec:proof}

In the course of the proofs worked out below, difference operators $\grad$, $\div$, and $\Delta$ are frequently used. Besides these particular operators, we also define the \emph{forward shift} operator $\SS$ on the space of complex sequences indexed by $\Z$ by equation
\begin{equation}
 \SS u_{n}:=u_{n+1}
\label{eq:def_S}
\end{equation}
for $n\in\Z$. Obviously, $\SS$ is invertible and $\SS^{-1}u_{n}=u_{n-1}$ for all $n\in\Z$. Recalling definition~\eqref{eq:def_grad_div}, we see that $\div=\SS-\I$ and $\grad=\I-\SS^{-1}$, where $\I$ stands for the identity operator. Particularly, we have identities $\div = \SS\grad=\grad\SS$ that will be used several times below. Moreover, it is clear that $\div$ and $\grad$ commute, i.e. $\div\grad u = \grad\div u\equiv\Delta u$ for all complex sequences $u$ and hence also all $u\in H^{\ell}$ for any $\ell$. On the other hand, subspaces $H^{\ell}$ are not preserved under the action of $\SS$ and so neither $\div$ ($\SS$ is a bijection of $H^{\ell}$ onto $H^{\ell-1}$). Due to these facts, although elementary, manipulations with the difference operators below require some caution.

\subsection{Proof of Theorem~\ref{thm:0}}\label{subsec:thm0_proof} 

Suppose $\gg\in H^{1}$, $\gg_{n}>0$ for all $n\geq 1$, and $u\in\mathcal{H}^{1}_{0}$. For any $n\geq2$, we have 
\begin{align*}
\left|\sqrt{\frac{\gg_{n-1}}{\gg_{n}}}u_{n}-\sqrt{\frac{\gg_{n}}{\gg_{n-1}}}u_{n-1}\right|^2&=
\frac{\gg_{n-1}}{\gg_{n}}|u_{n}|^{2}+\frac{\gg_{n}}{\gg_{n-1}}|u_{n-1}|^{2}-2\Re(u_{n}u_{n-1})\\
&=|\grad u_{n}|^{2}-\frac{\grad\gg_{n}}{\gg_{n}}|u_{n}|^{2}+\frac{\grad\gg_{n}}{\gg_{n-1}}|u_{n-1}|^{2}.
\end{align*}
Multiplying both sides by $V_{n}$ and summing over $n$ from $2$ to $\infty$, we obtain
\begin{align*}
\sum_{n=2}^{\infty}& V_n \left|\sqrt{\frac{\gg_{n-1}}{\gg_{n}}}u_{n}-\sqrt{\frac{\gg_{n}}{\gg_{n-1}}}u_{n-1}\right|^2 \\
&\hskip60pt=\sum_{n=2}^{\infty}V_{n}|\grad u_{n}|^{2}-\sum_{n=2}^{\infty}V_{n}\frac{\grad\gg_{n}}{\gg_{n}}|u_{n}|^{2}+\sum_{n=1}^{\infty}V_{n+1}\frac{\grad\gg_{n+1}}{\gg_{n}}|u_{n}|^{2} \\
&\hskip60pt=\sum_{n=2}^{\infty}V_{n}|\grad u_{n}|^{2}+\sum_{n=1}^{\infty}\frac{\div(V\grad\gg)_{n}}{\gg_{n}}|u_{n}|^{2}+V_{1}\frac{\grad\gg_{1}}{\gg_{1}}|u_{1}|^{2}.
\end{align*}
By assumptions, $u_{0}=\gg_{0}=0$. Therefore 
\[
V_{1}\frac{\grad\gg_{1}}{\gg_{1}}|u_{1}|^{2}=V_{1}|\grad u_{1}|^{2}
\]
and we get
\[
\sum_{n=2}^{\infty} V_n \left|\sqrt{\frac{\gg_{n-1}}{\gg_{n}}}u_{n}-\sqrt{\frac{\gg_{n}}{\gg_{n-1}}}u_{n-1}\right|^2=\sum_{n=1}^{\infty}V_{n}|\grad u_{n}|^{2}+\sum_{n=1}^{\infty}\frac{\div(V\grad\gg)_{n}}{\gg_{n}}|u_{n}|^{2}.
\]
A shift of the summation index $n$ by one on the left now yields the claim of Theorem~\ref{thm:0}. 

\qed

\subsection{Proof of Theorem~\ref{thm:1}}\label{subsec:thm1_proof}

The proof proceeds by a two-step induction in~$\ell\in\N$:
\begin{enumerate}[a)]
\item We verify Theorem~\ref{thm:1} for $\ell=1,2$.
\item Assuming Theorem~\ref{thm:1} to hold for all $\ell\leq 2m$, where $m\in\N$, we prove it for $\ell=2m+1$.
\item Assuming Theorem~\ref{thm:1} to hold for all $\ell\leq 2m+1$, we prove it also for $\ell=2m+2$.
\end{enumerate}
The reason to treat even and odd indices $\ell$ separately stems from the fact that, when lowering a half-integer power of the discrete Laplacian, $\grad$ or $\div$ pop up depending on the parity of~$\ell$ because we have 
\[
(-\Delta)^{\ell/2}=\begin{cases}
\grad(-\Delta)^{(\ell-1)/2}& \quad\mbox{ if } \ell \mbox{ is odd,}\\
-\div(-\Delta)^{(\ell-1)/2}& \quad\mbox{ if } \ell \mbox{ is even.}
\end{cases}
\]
Since the resulting differences are subtle, parts (b) and (c) of the proof are analogical and therefore the proof of~(c) is only briefly indicated.

a) For $\ell=1$, Theorem~\ref{thm:1} coincides with the special case of Theorem~\ref{thm:0} with $V\equiv1$. Suppose $\ell=2$ and~\eqref{eq:assum_A1}, i.e. $\gg\in H^{2}$, $\gg_{n}>0$ for all $n\geq2$,  and $\div\gg_{n}>0$ for all $n\geq1$. Clearly, $\div\gg\in H^{1}$ and so we may apply Theorem~\ref{thm:0} with $\gg$ replaced by $\div\gg$, $u$ replaced by $\div u$, and $V\equiv1$, from which it follows that
\[
\sum_{n=1}^{\infty}|\Delta u_{n}|^{2}=\sum_{n=1}^{\infty}|\grad \div u_{n}|^{2}=-\sum_{n=1}^{\infty}\frac{\Delta\div\gg_{n}}{\div\gg_{n}}|\div u_{n}|^{2}+\mathcal{R}_{1}^{(2)}(\gg;u)\\
\]
for all $u\in\mathcal{H}^{2}_{0}$, where $\mathcal{R}_{1}^{(2)}(\gg;u)$ is defined in~\eqref{eq:def_remainder_R}. Bearing in mind that $\div u=\grad\SS u$ (recall~\eqref{eq:def_S}), we apply Theorem~\ref{thm:0} once more, this time with $\gg$ replaced by $\SS\gg$, $u$ replaced by $\SS u$, and $V_{n}:=-\Delta\div\gg_{n}/\div\gg_{n}$ to the first term on the right getting the identity
\[
\sum_{n=1}^{\infty}|\Delta u_{n}|^{2}=-\sum_{n=1}^{\infty}\frac{\div(V\grad \SS\gg)_{n}}{\SS\gg_{n}}|\SS u_{n}|^{2}+\mathcal{R}_{0}^{(2)}(\gg;u)+\mathcal{R}_{1}^{(2)}(\gg;u)
\]
for all $u\in\mathcal{H}^{2}_{0}$, with $\mathcal{R}_{0}^{(2)}(\gg;u)$ given again by the general definition from~\eqref{eq:def_remainder_R}. Taking also into account that 
\[
-\div(V\grad\SS\gg)=\div\left(\frac{\Delta\div\gg}{\div\gg} \div\gg\right)=\SS\Delta^{2}\gg,
\]
we obtain
\[
\sum_{n=1}^{\infty}|\Delta u|^{2}=\sum_{n=2}^{\infty}\frac{\Delta^{2}\gg_{n}}{\gg_{n}}|u_{n}|^{2}+\mathcal{R}_{0}^{(2)}(\gg;u)+\mathcal{R}_{1}^{(2)}(\gg;u)
\]
for all $u\in\mathcal{H}^{2}_{0}$, which is the identity~\eqref{eq:hrb_id_init} for $\ell=2$.

b) Suppose $m\in\N$ and the implication of Theorem~\ref{thm:1} to hold for all $\ell\leq2m$. Let $\gg\in H^{2m+1}$ fulfill the assumption~\eqref{eq:assum_A1} for $\ell=2m+1$ and $u\in\mathcal{H}_{0}^{2m+1}$. We show that~\eqref{eq:hrb_id_init} holds true for $\ell=2m+1$.

We have
\begin{align*}
 \sum_{n=\lceil\frac{2m+1}{2}\rceil}^{\infty}\left|(-\Delta)^{(2m+1)/2}u_{n}\right|^{2}&=
 \sum_{n=m+1}^{\infty}\left|\grad(-\Delta)^{m}u_{n}\right|^{2}=
 \sum_{n=m+1}^{\infty}\left|(-\Delta)^{m}\grad u_{n}\right|^{2}\\
 &=\sum_{n=m}^{\infty}\left|(-\Delta)^{m}\div u_{n}\right|^{2}.
\end{align*}
Since $\div u\in\mathcal{H}_{0}^{2m}$ and $\div\gg$ satisfies assumption~\eqref{eq:assum_A1} for $\ell=2m$, we may apply the induction hypothesis and obtain
\[
\sum_{n=m}^{\infty}\left|(-\Delta)^{m}\div u_{n}\right|^{2}=\sum_{n=2m}^{\infty}\frac{\Delta^{2m}\div\gg_{n}}{\div\gg_{n}}|\div u_{n}|^{2}+\sum_{k=0}^{2m-1}\mathcal{R}_{k}^{(2m)}(\div\gg;\div u).
\]
It follows from formula~\eqref{eq:def_remainder_R} that 
\[
\mathcal{R}_{k}^{(2m)}(\div\gg;\div u)=\mathcal{R}_{k+1}^{(2m+1)}(\gg;u).
\]
Shifting also the index in the first sum, we find that  
\begin{align}
\sum_{n=\lceil\frac{2m+1}{2}\rceil}^{\infty}\left|(-\Delta)^{(2m+1)/2}u_{n}\right|^{2}=\sum_{n=1}^{\infty}\frac{\SS^{2m-1}\Delta^{2m}\div\gg_{n}}{\SS^{2m-1}\div\gg_{n}}|\grad \SS^{2m} u_{n}|^{2}+\sum_{k=1}^{2m}\mathcal{R}_{k}^{(2m+1)}(\gg;u).\nonumber\\
\label{eq:semi-id_inducb_inproof}
\end{align}

Next, we apply Theorem~\ref{thm:0} with $u$ replaced by $\SS^{2m}u\in\mathcal{H}_{0}^{1}$, $\gg$ replaced by $\SS^{2m}\gg\in H^{1}$, and
\[
V_{n}:=\frac{\SS^{2m-1}\Delta^{2m}\div\gg_{n}}{\SS^{2m-1}\div\gg_{n}},
\]
to the first term on the right in~\eqref{eq:semi-id_inducb_inproof}. It results in the equality 
\begin{align*}
\sum_{n=1}^{\infty}\frac{\SS^{2m-1}\Delta^{2m}\div\gg_{n}}{\SS^{2m-1}\div\gg_{n}}|\grad \SS^{2m} u_{n}|^{2}&=-\sum_{n=1}^{\infty}\frac{\div(V\grad\SS^{2m}\gg)_{n}}{\SS^{2m}\gg_{n}}|\SS^{2m}u_{n}|^{2}\\
&+\sum_{n=1}^{\infty}V_{n+1}\left|\sqrt{\frac{\SS^{2m}\gg_{n}}{\SS^{2m}\gg_{n+1}}}\SS^{2m}u_{n+1}-\sqrt{\frac{\SS^{2m}\gg_{n+1}}{\SS^{2m}\gg_{n}}}\SS^{2m}u_{n}\right|^{2}.
\end{align*}
Taking also into account that 
\begin{align*}
-\div(V\grad\SS^{2m}\gg)&=-\div\left(\frac{\SS^{2m-1}\Delta^{2m}\div\gg}{\SS^{2m-1}\div\gg}\SS^{2m-1}\div\gg\right)=-\div\SS^{2m-1}\Delta^{2m}\div\gg\\
&=\SS^{2m}(-\Delta)^{2m+1}\gg,
\end{align*}
we arrive, after shifting indices, at the formula
\begin{align*}
\sum_{n=1}^{\infty}\frac{\SS^{2m-1}\Delta^{2m}\div\gg_{n}}{\SS^{2m-1}\div\gg_{n}}|\grad \SS^{2m} u_{n}|^{2}&=\sum_{n=2m+1}^{\infty}\frac{(-\Delta)^{2m+1}\gg_{n}}{\gg_{n}}|u_{n}|^{2}\\
&+\sum_{n=2m+1}^{\infty}\frac{\Delta^{2m}\div\gg_{n}}{\div\gg_{n}}\left|\sqrt{\frac{\gg_{n}}{\gg_{n+1}}}u_{n+1}-\sqrt{\frac{\gg_{n+1}}{\gg_{n}}}u_{n}\right|^{2}.
\end{align*}
By~\eqref{eq:def_remainder_R}, the last term coincides with $\mathcal{R}_{0}^{(2m+1)}(\gg;u)$. Thus, when combined with~\eqref{eq:semi-id_inducb_inproof}, we obtain identity~\eqref{eq:hrb_id_init} for $\ell=2m+1$.

c) Suppose Theorem~\ref{thm:1} holds for all $\ell\leq2m+1$, $\gg\in H^{2m+2}$ satisfies~\eqref{eq:assum_A1} for $\ell=2m+2$, and $u\in\mathcal{H}_{0}^{2m+2}$. We have
\[
 \sum_{n=\lceil\frac{2m+2}{2}\rceil}^{\infty}\left|(-\Delta)^{(2m+2)/2}u_{n}\right|^{2}
 =\sum_{n=\lceil\frac{2m+1}{2}\rceil}^{\infty}\left|(-\Delta)^{(2m+1)/2}\div u_{n}\right|^{2}.
\]
Since $\div u\in\mathcal{H}_{0}^{2m+1}$ and $\div \gg\in H^{2m+1}$ satisfies~\eqref{eq:assum_A1} for $\ell=2m+1$, we may apply the induction hypothesis with $u$ replaced by $\div u$ and $\gg$ replaced by $\div\gg$ getting
\begin{align*}
\sum_{n=\lceil\frac{2m+1}{2}\rceil}^{\infty}\left|(-\Delta)^{(2m+1)/2}\div u_{n}\right|^{2}&=\sum_{n=2m+1}^{\infty}\frac{(-\Delta)^{2m+1}\div\gg_{n}}{\div\gg_{n}}|\div u_{n}|^{2}\\
&+\sum_{k=0}^{2m}\mathcal{R}_{k}^{(2m+1)}(\div\gg;\div u).
\end{align*}
The rest of the proof proceeds analogically as in part~(b). The proof of Theorem~\ref{thm:1} is complete.

\qed

\subsection{Proof of Theorem~\ref{thm:2}} 

The claim is an immediate consequence of identity~\eqref{eq:hrb_id_init}. It suffices to notice that assumption~\eqref{eq:assum_A2} together with~\eqref{eq:assum_A1} guarantee $\mathcal{R}_{k}^{(\ell)}(\gg;u)\geq0$ for all $k=0,1,\dots,\ell-1$. Assumption~\eqref{eq:assum_A2'} means nothing but the non-negativity of $\rho(\gg)$.
The proof of Theorem~\ref{thm:2} is complete.

\qed

\subsection{Proof of Theorem~\ref{thm:3}}

We check that, under the respective assumptions, nonnegative weight $\rho(\gg):=(-\Delta)^{\ell}\gg/\gg$ possesses the three properties from Definition~\ref{def:optim}: a) criticality; b) optimality near infinity; c) non-attainability.

We will need three auxiliary claims. For more concise formulas, we introduce another averaging difference operator. Recall the definition of the shift~\eqref{eq:def_S} and define $\M:=(\I+\SS)/2$, i.e.
\[
\M u_{n}:=\frac{u_{n}+u_{n+1}}{2}
\]
for any complex sequence $u$ indexed by $\Z$.

\begin{lem}\label{lem:opt1}
Let $k\in\N$, $\alpha\in\R\setminus\N_{0}$, and $g$ be a sequence with the asymptotic expansion
\[
 g_{n}=\sum_{j=0}^{k}a_{j}n^{\alpha-j}+\bigO\left(n^{\alpha-k-1}\right), \quad n\to\infty,
\]
for some $a_{j}\in\R$ and $a_{0}\neq0$. Then the following claims hold true.
\begin{enumerate}[{\upshape i)}]
\item For all $m\in\Z$, there are real coefficients $a_{j}^{(m)}$ such that 
\[
 \SS^{m} g_{n}=a_{0}n^{\alpha}+\sum_{j=1}^{k}a_{j}^{(m)}n^{\alpha-j}+\bigO\left(n^{\alpha-k-1}\right), \quad n\to\infty.
\]
\item For all $m\in\N_{0}$, there are real coefficients $b_{j}^{(m)}$ with $b_{0}^{(m)}\neq0$ such that 
\[
 \M^{m} g_{n}=\sum_{j=0}^{k}b_{j}^{(m)}n^{\alpha-j}+\bigO\left(n^{\alpha-k-1}\right), \quad n\to\infty.
\]
\item For all $m\in\{0,1,\dots,k\}$, there are real coefficients $c_{j}^{(m)}$ with $c_{m}^{(m)}\neq0$ such that 
\[
 \div^{m} g_{n}=\sum_{j=m}^{k}c_{j}^{(m)}n^{\alpha-j}+\bigO\left(n^{\alpha-k-1}\right), \quad n\to\infty.
\]
\end{enumerate}
\end{lem}

The proof of Lemma~\ref{lem:opt1} is an easy exercise and is therefore omitted. The next auxiliary claim is a higher-order variant of the Mean Value Theorem.

\begin{lem}\label{lem:opt2}
Let $n\in\Z$, $N\in\N$, and $g$ be a continuous function on $[n,n+N]$ of class $C^{N}$ in $(n,n+N)$. Then there exists $\xi\in(n,n+N)$ such that
\[
 \div^{N}g_{n}=g^{(N)}(\xi),
\] 
where we denoted $g_{n}:=g(n)$.
\end{lem}

Lemma~\ref{lem:opt2} is most likely known. In order not to distract the reader from our main purpose, we postpone its proof to the Appendix. The last auxiliary identity is a Leibnitz formula for the discrete divergence. The multiplication of sequences is to be understood point-wise.

\begin{lem}\label{lem:opt3}
For all $m\in\N_{0}$ and sequences $u$ and $v$, we have
\[
 \div^{m}(uv)=\sum_{j=0}^{m}\binom{m}{j}\left(\div^{j}\M^{m-j}u\right)\left(\div^{m-j}\M^{j}v\right).
\]
\end{lem}

Proof of Lemma~\ref{lem:opt3} proceeds by induction in $m$, its details are postponed to the Appendix.
Now, we start proving Theorem~\ref{thm:3}.

i) \emph{Proof of the criticality:} Suppose $\tilde{\rho}=\{\tilde{\rho}_{n}\}_{n=\ell}^{\infty}$ is such that inequality~\eqref{eq:hrb_ineq_disc} holds with $\rho$ replaced by $\tilde{\rho}$ and $\tilde{\rho}_{n}\geq\rho_{n}(\gg)$ for all $n\geq\ell$. Using identity~\eqref{eq:hrb_id_init} of Theorem~\ref{thm:1} together with the Hardy--Rellich--Birman inequality for $\tilde{\rho}$, we find that
\begin{equation}
 0\leq\sum_{n=\ell}^{\infty}\left(\tilde{\rho}_{n}-\rho_{n}(\gg)\right)|u_{n}|^{2}\leq\sum_{k=0}^{\ell-1}\mathcal{R}_{k}^{(\ell)}(\gg;u)
\label{eq:tilde_rho_ineq_critic_inproof}
\end{equation}
for all $u\in\mathcal{H}_{0}^{\ell}$. 

Assumptions~\eqref{eq:assum_A1} and~\eqref{eq:assum_A2} imply that each term in the sum from definition~\eqref{eq:def_remainder_R} of remainders $\mathcal{R}_{k}^{(\ell)}(\gg;u)$ is nonnegative for all $0\leq k<\ell$. Then, using the definition of $R_{k}^{(\ell)}(\gg)$ from~\eqref{eq:def_Ru_critic_inproof}, we have
\[
 \mathcal{R}_{k}^{(\ell)}(\gg;u)=\sum_{n=\ell-k}^{\infty}\left|R_{k}^{(\ell)}(\gg)u_{n}\right|^{2}
\]
for all $0\leq k<\ell$. An important fact is that the remainders are simultaneously annihilated if $u=\gg$ since $R_{k}^{(\ell)}(\gg)\gg_{n}=0$ for all $n\geq\ell-k$ and $0\leq k<\ell$.
However, we cannot directly substitute $u=\gg$ into~\eqref{eq:tilde_rho_ineq_critic_inproof} and conclude from here that $\tilde{\rho}=\rho(\gg)$ since $\gg$ need not be compactly supported, i.e. not an element of $\mathcal{H}_{0}^{\ell}$. This is an issue which is to be overcome by a suitable regularization of $\gg$.

Fix arbitrary $\varepsilon\in(0,1/2)$ and a smooth function $\eta$ such that $\eta\equiv0$ on $(-\infty,\varepsilon)$ and $\eta\equiv1$ on $(1-\varepsilon,\infty)$. Then for any $N\geq2$, we put $u^{N}:=\xi^{N}\gg$, where $\xi^{N}_{n}:=\xi^{N}(n)$ is the cut-off sequence defined by
\begin{equation}
 \xi^{N}(x):=\begin{cases}
 1 & \mbox{ if } x\leq N,\\
 \eta\!\left(\frac{2\log N-\log x}{\log N}\right) & \mbox{ if } N<x\leq N^{2},\\
 0 & \mbox{ if } x> N^{2}.\\
 \end{cases}
\label{eq:def_xi_critic_inproof}
\end{equation}
Notice that $\xi^{N}\to1$ and hence $u^{N}\to\gg$ point-wise as $N\to\infty$.
With this choice of $u^{N}\in\mathcal{H}_{0}^{\ell}$, we will show that for all $0\leq k<\ell$,
\begin{equation}
 0\leq \mathcal{R}_{k}^{(\ell)}\left(\gg;u^{N}\right)\lesssim\frac{1}{\log N},
\label{eq:R_bound_critic_inproof}
\end{equation}
where $\lesssim$ means the inequality $\leq$ up to a $N$-independent multiplicative constant. From this, inequality~\eqref{eq:tilde_rho_ineq_critic_inproof}, and Fatou's lemma, we infer that
\[
 0=\liminf_{N\to\infty}\sum_{n=\ell}^{\infty}\left(\tilde{\rho}_{n}-\rho_{n}(\gg)\right)|u_{n}^{N}|^{2}\geq\sum_{n=\ell}^{\infty}\left(\tilde{\rho}_{n}-\rho_{n}(\gg)\right)\lim_{N\to\infty}|u_{n}^{N}|^{2}=\sum_{n=\ell}^{\infty}\left(\tilde{\rho}_{n}-\rho_{n}(\gg)\right)\gg_{n}^{2}.
\] 
Since all terms in the last sum are nonnegative and $\gg_{n}>0$ for all $n\geq\ell$ by~\eqref{eq:assum_A1}, we conclude that $\tilde{\rho}_{n}=\rho_{n}(\gg)$ for all $n\geq\ell$, and the proof of the criticality of $\rho(\gg)$ will be complete.

It remains to verify inequality~\eqref{eq:R_bound_critic_inproof} which is done in the rest of the part (a) of this proof. We substitute for $u=u^{N}=\xi^{N}\gg$ into~\eqref{eq:def_Ru_critic_inproof} and inspect the two factors -- the prefactor term and the term in the brackets -- separately. For the first factor, using the assumption~\eqref{eq:assum_A3} and claims (i) and (iii) of Lemma~\ref{lem:opt1}, we find that
\[
\sqrt{\frac{(-\Delta)^{\ell-1-k}\div^{k+1}\gg_{n}}{\div^{k+1}\gg_{n}}}=\sqrt{\frac{\div^{2\ell-k-1}\SS^{k+1-\ell}\gg_{n}}{\div^{k+1}\gg_{n}}}\lesssim\sqrt{\frac{n^{\ell-s-(2\ell-k-1)}}{n^{\ell-s-(k+1)}}}=n^{k+1-\ell}
\]
for any $n\geq\ell$ and $0\leq k<\ell$, where the unspecified constant is $n$-independent but may depend on $k$ and $\ell$.

Similarly, by Lemma~\ref{lem:opt1} and~\eqref{eq:assum_asympt_g}, we find that
\[
\sqrt{\frac{\div^{k}\gg_{n}}{\div^{k}\gg_{n+1}}}\lesssim1,
\]
and therefore the second factor
\[
\sqrt{\frac{\div^{k}\gg_{n}}{\div^{k}\gg_{n+1}}}\,\left|\,\div^{k}\left(\xi^{N}\gg\right)_{n+1}-\frac{\div^{k}\gg_{n+1}}{\div^{k}\gg_{n}}\div^{k}\left(\xi^{N}\gg\right)_{n}\right|
\]
is majorized by
\[
\left|\,\div^{k}\left(\xi^{N}\gg\right)_{n+1}-\frac{\div^{k}\gg_{n+1}}{\div^{k}\gg_{n}}\div^{k}\left(\xi^{N}\gg\right)_{n}\right|
\]
up to a multiplicative constant independent of $n$ and $N$.
Next, we apply Lemma~\ref{lem:opt3} in the last expression and rewrite it as
\begin{align*}
&\Bigg|\sum_{j=0}^{k}\binom{k}{j}(\div^{j}\M^{k-j}\gg)_{n+1}(\div^{k-j}\M^{j}\xi^{N})_{n+1}\\
&\hskip150pt-\frac{\div^{k}\gg_{n+1}}{\div^{k}\gg_{n}}\sum_{j=0}^{k}\binom{k}{j}(\div^{j}\M^{k-j}\gg)_{n}(\div^{k-j}\M^{j}\xi^{N})_{n}
\Bigg|\\
&\leq\sum_{j=0}^{k}\binom{k}{j}(\div^{j}\M^{k-j}\gg)_{n+1}\left|(\div^{k-j}\M^{j}\xi^{N})_{n+1}
-X_{n}^{(k,j)}(\gg)\,(\div^{k-j}\M^{j}\xi^{N})_{n}
\right|,
\end{align*}
where we denoted
\[
 X_{n}^{(k,j)}(\gg):=\frac{\div^{k}\gg_{n+1}}{\div^{k}\gg_{n}}\frac{\div^{j}\M^{k-j}\gg_{n}}{\div^{j}\M^{k-j}\gg_{n+1}}.
\]
Yet another application of~\eqref{eq:assum_asympt_g} and formulas from Lemma~\ref{lem:opt1} yields estimates
\[
 (\div^{j}\M^{k-j}\gg)_{n+1}\lesssim n^{\ell-s-j}
\quad\mbox{ and }\quad
 X_{n}^{(k,j)}(\gg)=1+p_{n}^{(k,j)},
\]
where $p_{n}^{(k,j)}=\bigO(1/n)$, for $n\to\infty$, i.e. $p_{n}^{(k,j)}\lesssim 1/n$. Altogether, we deduce the upper bound
\begin{equation}
 \left|R_{k}^{(\ell)}(\gg)u^{N}_{n}\right|\lesssim
 n^{k+1-s}\sum_{j=0}^{k}\binom{k}{j}n^{-j}\left(\left|(\div^{k+1-j}\M^{j}\xi^{N})_{n}\right|+\left|p_{n}^{(k,j)}\right|\left|(\div^{k-j}\M^{j}\xi^{N})_{n}\right|\right)
\label{eq:R_semi-bound_critic_inproof}
\end{equation}
for all $N\geq2$, $n\geq\ell$, and $0\leq k<\ell$, where the unspecified multiplicative constant does not depend on $n$ and $N$.

Further, we estimate also the difference expressions from~\eqref{eq:R_semi-bound_critic_inproof} applied to $\xi^{N}$. To this end, first note that for all $x\in[N,N^{2}]$, we have
\[
 (\xi^{N})'(x)=-\eta'\!\left(\frac{2\log N-\log x}{\log N}\right)\frac{1}{x\log N},
\]
from which we readily get the estimate 
\[
|(\xi^{N})'(x)|\lesssim\frac{1}{x\log N}
\]
where the unspecified constant can be taken as $\max_{x\in[0,1]}|\eta'(x)|$. By induction, it is straightforward to generalize this bound to higher-order derivatives
\begin{equation}
|(\xi^{N})^{(m)}(x)|\lesssim\frac{1}{x^{m}\log N},
\label{eq:xi_der_bound_critic_inproof}
\end{equation}
which holds true for any $m\in\N$ and $x>0$. Then using Lemma~\ref{lem:opt2}, we obtain the estimate
\[
 \left|\div^{m}\xi_{n}^{N}\right|\lesssim\frac{1}{n^{m}\log N}
\]
which is true for all $m,n\in\N$ and $N\geq2$. As the right-hand side is a decreasing function of $n$, we also have 
\[
 \left|\div^{m}\xi_{n+j}^{N}\right|\lesssim\frac{1}{n^{m}\log N}
\]
for any $j\in\N_{0}$, from which we infer the needed estimates in~\eqref{eq:R_semi-bound_critic_inproof} getting
\[
 \left|R_{k}^{(\ell)}(\gg)u^{N}_{n}\right|\lesssim
 n^{k+1-s}\sum_{j=0}^{k}\binom{k}{j}n^{-j}\left[\frac{1}{n^{k+1-j}\log N}+\frac{1}{n}\frac{1}{n^{k-j}\log N}\right]\lesssim\frac{1}{n^{s}\log N}\leq\frac{1}{\sqrt{n}\log N}
\]
for all $N\geq2$, $n\geq\ell$, and $0\leq k<\ell$, where we have used the assumption $s\geq1/2$ from~\eqref{eq:assum_A3}. Finally, using the last estimate, we obtain
\[
 \mathcal{R}_{k}^{(\ell)}\left(\gg;u^{N}\right)=\sum_{n=N-k}^{N^{2}}\left|R_{k}^{(\ell)}(\gg)u_{n}^{N}\right|^{2}\lesssim\frac{1}{\log^{2}N}\sum_{n=N-k}^{N^{2}}\frac{1}{n}\lesssim\frac{1}{\log^{2}N}\int_{N}^{N^{2}}\frac{\dd n}{n}=\frac{1}{\log N}.
\]
for all $0\leq k<\ell$ and $N\geq\ell$, arriving at the desired upper bound~\eqref{eq:R_bound_critic_inproof}.

ii) \emph{Proof of the optimality near infinity:}
Fix $M\in\N$. Recalling~\eqref{eq:def_opt_near_inf} together with the identity~\eqref{eq:hrb_id_init}, the optimality of $\rho(\gg)$ near infinity will be verified if we find a sequence of elements $u^{N}\in\mathcal{H}_{0}^{M}\setminus\{0\}$ such that
\begin{equation}
 \lim_{N\to\infty}\frac{\sum\limits_{k=0}^{\ell-1}\mathcal{R}_{k}^{(\ell)}\left(\gg;u^{N}\right)}{\sum\limits_{n=\ell}^{\infty}\rho_{n}(\gg)|u_{n}^{N}|^{2}}=0;
\label{eq:lim_ratio_opt_near_inf_inproof}
\end{equation}
it will follow from~\eqref{eq:denom_bound_opt_near_inf_inproof} below that, with our choice of $u^{N}$, the denominator in~\eqref{eq:lim_ratio_opt_near_inf_inproof} does not vanish for all $N$ sufficiently large.
Such a sequence can be chosen as $u^{N}:=\xi^{N}\gg$, where the regularizing sequence $\xi^{N}$ is a slight modification of~\eqref{eq:def_xi_critic_inproof}. This time, we put
\[
 \xi^{N}(x):=\begin{cases}
 0 & \mbox{ if } x\leq N,\\
 \eta\!\left(\frac{\log x-\log N}{\log N}\right) & \mbox{ if } N<x\leq N^{2},\\
 1 & \mbox{ if } N^{2}<x\leq 2N^{2},\\
  \eta\!\left(\frac{\log(2N^{3})-\log x}{\log N}\right) & \mbox{ if } 2N^{2}<x\leq 2N^{3},\\
 0 & \mbox{ if } x> 2N^{3},\\
 \end{cases}
\]
where the function $\eta$ is the same as in~\eqref{eq:def_xi_critic_inproof}. Then $u^{N}\in\mathcal{H}_{0}^{M}\setminus\{0\}$ for all $N\geq M$. 

With this new $\xi^{N}$, inequality~\eqref{eq:xi_der_bound_critic_inproof} still holds for any $m\in\N$ and $x>0$ and the same estimates as in the proof of part~(a) apply. Consequently, we find that
\begin{align}
 \mathcal{R}_{k}^{(\ell)}\left(\gg;u^{N}\right)&=\sum_{n=N-k}^{N^{2}}\left|R_{k}^{(\ell)}(\gg)u_{n}^{N}\right|^{2}+\sum_{n=2N^{2}-k}^{2N^{3}}\left|R_{k}^{(\ell)}(\gg)u_{n}^{N}\right|^{2} \nonumber\\
 &\lesssim\frac{1}{\log^{2}N}\left(\sum_{n=N-k}^{N^{2}}\frac{1}{n}+\sum_{n=2N^{2}-k}^{2N^{3}}\frac{1}{n}\right)\lesssim\frac{1}{\log N} \label{eq:nomin_bound_opt_near_inf_inproof}
\end{align}
for all $0\leq k<\ell$ and $N\geq\ell$.

Next, we estimate the denominator in~\eqref{eq:lim_ratio_opt_near_inf_inproof} from below. 
Using the assumption~\eqref{eq:assum_A3'} and Lemma~\ref{lem:opt1}, one readily verifies that
\begin{equation}
\rho_{n}(\gg)=\frac{(-\Delta)^{\ell}\gg_{n}}{\gg_{n}}=\frac{\beta}{n^{2\ell}}+\bigO\left(\frac{1}{n^{2\ell+1}}\right), \; \mbox{ as } n\to\infty,
\label{eq:rho_lower_bound_inproof}
\end{equation}
with a constant $\beta\neq0$ (a more precise calculation reminiscent of those made in the proof of Theorem~\ref{thm:5} below yields the exact value $\beta=(1/2)_{\ell}^{2}$). Thus, yet another use of~\eqref{eq:assum_A3'} together with expansion~\eqref{eq:rho_lower_bound_inproof} yields
\begin{equation}
\sum_{n=\ell}^{\infty}\rho_{n}(\gg)|u_{n}^{N}|^{2}\geq\sum_{n=N^{2}}^{2N^{2}}\rho_{n}(\gg)\gg_{n}^{2}\gtrsim\sum_{n=N^{2}}^{2N^{2}}\frac{1}{n}\geq\int_{N^{2}}^{2N^{2}}\frac{\dd n}{n}=\log 2
 \label{eq:denom_bound_opt_near_inf_inproof}
\end{equation}
for all $N$ sufficiently large.
Estimates~\eqref{eq:nomin_bound_opt_near_inf_inproof} and~\eqref{eq:denom_bound_opt_near_inf_inproof} imply~\eqref{eq:lim_ratio_opt_near_inf_inproof}.

iii) \emph{Proof of the non-attainability:} For the proof of attainability, we first show that under certain assumptions, the identity from Theorem~\ref{thm:1} extends from $\mathcal{H}_{0}^{\ell}$ to all sequences of $H^{\ell}$ for which the left-hand side of~\eqref{eq:hrb_id_init} is finite. Although we use this extension only for the proof of the non-attainability, the claim of the following lemma can be of independent interest.
 
\begin{lem}\label{lem:extend}
Under the assumptions~\eqref{eq:assum_A1}, \eqref{eq:assum_A2}, and \eqref{eq:assum_A2'} with the strict inequality, i.e. $(-\Delta)^{\ell}\gg_{n}>0$ for all $n\geq\ell$, the identity~\eqref{eq:hrb_id_init} extends from $\mathcal{H}_{0}^{\ell}$ to all sequences from the space 
\[
\mathcal{D}^{\ell}:=\left\{u\in H^{\ell} \;\big|\; \|(-\Delta)^{\ell/2}u\|<\infty\right\}.
\]
\end{lem}

\begin{proof}[Proof of Lemma~\ref{lem:extend}]

\emph{Step~1:} We show that the range of $(-\Delta)^{\ell/2}|_{\mathcal{H}_{0}^{\ell}}$ is dense in $\mathcal{H}^{\lceil \ell/2\rceil}$. It follows from the definition of $(-\Delta)^{\ell/2}$ that, for $u\in \mathcal{H}^{\ell}_{0}$, $(-\Delta)^{\ell/2}u_{n}=0$ for all $n<\lceil\ell/2\rceil$, i.e. $(-\Delta)^{\ell/2}u\in H^{\lceil\ell/2\rceil}$, and $(-\Delta)^{\ell/2}u\in\ell^{2}(\Z)$ by the boundedness of $(-\Delta)^{\ell/2}$. Suppose that $v\in\mathcal{H}^{\lceil \ell/2\rceil}$ satisfies 
\[
 \langle v,(-\Delta)^{\ell/2}u\rangle =0
\]
for all $u\in\mathcal{H}_{0}^{\ell}$. We will show that it follows $v=0$.

According to the parity of $\ell$, we distinguish two cases. Let $\ell=2m$ for some $m\in\N$. Then
\[
\langle v,(-\Delta)^{\ell/2}u\rangle=\langle v,(-\Delta)^{m}u\rangle=\langle(-\Delta)^{m}v,u\rangle=0
\]
for all $u\in\mathcal{H}_{0}^{2m}$. By taking $u=\delta_{n}$, with $n\geq2m$, we observe that $v$ solves the difference equation
\[
 (-\Delta)^{m}v_{n}=\sum_{j=-m}^{m}\binom{2m}{m+j}(-1)^{j}v_{n+j}=0, \quad \forall n\geq 2m.
\]
It is easy to show that the fundamental system of the above linear difference equation with constant coefficients consists of functions $1,n,\dots,n^{2m-1}$. Therefore we find that 
\[
 v_{n}=\sum_{j=0}^{2m-1}c_{j}n^{j}, \quad \forall n\geq m,
\]
with some $c_{j}\in\C$. Taking also into account that $v\in\ell^{2}(\Z)$, we conclude that $c_{j}=0$ for all $j=0,1,\dots,2m-1$, and so $v=0$.

The case $\ell=2m+1$ for $m\in\N_{0}$ is to be treated similarly. In this case, we have
\[
\langle v,(-\Delta)^{\ell/2}u\rangle=\langle v,\grad(-\Delta)^{m}u\rangle=-\langle(-\Delta)^{m}\div v,u\rangle=0
\]
for all $u\in\mathcal{H}_{0}^{2m+1}$. By taking $u=\delta_{n}$ with $n\geq 2m+1$, we find that $v$ is a solution of the linear difference equation with constant coefficients
\[
 -(-\Delta)^{m}\div v_{n}=\sum_{j=-m}^{m+1}\binom{2m+1}{m+j}(-1)^{j}v_{n+j}=0, \quad \forall n\geq 2m+1,
\]
whose general solution is the linear combination 
\[
 v_{n}=\sum_{j=0}^{2m}c_{j}n^{j}, \quad \forall n\geq m+1.
\]
Since $v\in\ell^{2}(\Z)$ we again conclude that $v=0$.

\emph{Step 2:} Pick arbitrary $v\in\mathcal{H}^{\lceil\ell/2\rceil}$. By Step~1, there exists sequence $u^{N}\in\mathcal{H}_{0}^{\ell}$ such that $(-\Delta)^{\ell/2}u^{N}\to v$, as $N\to\infty$, in the metric of $\ell^{2}(\Z)$. In particular, the sequence $\{(-\Delta)^{\ell/2}u^{N}\}_{N=1}^{\infty}$ is Cauchy in $\ell^{2}(\Z)$. Assumptions \eqref{eq:assum_A1}, \eqref{eq:assum_A2}, and~\eqref{eq:assum_A2'} guarantee non-negativity of $\rho(\gg)$ as well as the reminder terms. Consequently, identity~\eqref{eq:hrb_id_init} implies inequalities
\[
 \|\sqrt{\rho(\gg)}(u^{N}-u^{M})\|\leq\|(-\Delta)^{\ell/2}(u^{N}-u^{M})\|
\]
and
\[
 \|R_{k}^{(\ell)}(\gg)(u^{N}-u^{M})\|\leq\|(-\Delta)^{\ell/2}(u^{N}-u^{M})\|
\]
for all $0\leq k<\ell$ and $M,N\in\N$. Thus, also sequences $\{\sqrt{\rho(\gg)}u^{N}\}_{N=1}^{\infty}$ and $\{R_{k}^{(\ell)}(\gg)u^{N}\}_{N=1}^{\infty}$, for all $0\leq k<\ell$, are Cauchy and so convergent in $\ell^{2}(\Z)$.

Let $w$ be the $\ell^{2}$-limit of $\sqrt{\rho(\gg)}u^{N}$ as $N\to\infty$. Clearly, $w\in\mathcal{H}^{\ell}$ and $\sqrt{\rho(\gg)}u^{N}\to w$ point-wise as $N\to\infty$. By the assumptions, $\rho_{n}(\gg)>0$ for all $n\geq\ell$ and therefore there exists $u\in H^{\ell}$ such that $w=\sqrt{\rho(\gg)}u$ and $u_{n}^{N}\to u_{n}$, as $N\to\infty$, for all $n\in\Z$. Moreover, the limits of the $\ell^{2}$-convergent sequences $\{(-\Delta)^{\ell/2}u^{N}\}_{N=1}^{\infty}$, $\{\sqrt{\rho(\gg)}u^{N}\}_{N=1}^{\infty}$, and $\{R_{k}^{(\ell)}(\gg)u^{N}\}_{N=1}^{\infty}$ have to coincide with their point-wise limits, which are $(-\Delta)^{\ell/2}u$, $\sqrt{\rho(\gg)}u$, and $R_{k}^{(\ell)}(\gg)u$, respectively. In particular, $v=(-\Delta)^{\ell/2} u$. Therefore we may pass to the limit $N\to\infty$ in the identity~\eqref{eq:hrb_id_init} with $u=u^{N}$, i.e. in
\[
 \|(-\Delta)^{\ell/2}u^{N}\|^{2}=\|\sqrt{\rho_{n}(\gg)}u^{N}\|^{2}+\sum_{k=0}^{\ell-1}\|R_{k}^{(\ell)}(\gg)u^{N}\|^{2},
\]
getting the equality 
\begin{equation}
 \|(-\Delta)^{\ell/2}u\|^{2}=\|\sqrt{\rho_{n}(\gg)}u\|^{2}+\sum_{k=0}^{\ell-1}\|R_{k}^{(\ell)}(\gg)u\|^{2}.
\label{eq:hrb_id_norm-form_inproof}
\end{equation}

\emph{Step~3:} Now, pick $\tilde{u}\in\mathcal{D}^{\ell}$ arbitrarily. Then $(-\Delta)^{\ell/2}\tilde{u}\in\mathcal{H}^{\lceil\ell/2\rceil}$ by the definition of $\mathcal{D}^{\ell}$. By Step~2 applied to $v:=(-\Delta)^{\ell/2}\tilde{u}$, we find $u\in H^{\ell}$ such that $v=(-\Delta)^{\ell/2}u$ and the identity~\eqref{eq:hrb_id_norm-form_inproof} holds. Therefore in order to finish the proof, it remains to check that the equality $(-\Delta)^{\ell/2}\tilde{u}=(-\Delta)^{\ell/2}u$ for two vectors $u,\tilde{u}\in H^{\ell}$ implies $u=\tilde{u}$. 

By linearity, it suffices to assume that $(-\Delta)^{\ell/2}w=0$ for $w\in H^{\ell}$ and conclude that $w=0$. This is immediate since $(-\Delta)^{\ell/2}w_{n}=0$ is a linear difference equation of order $\ell$ with non-zero constant coefficients and $w_{n}=0$ for $n<\ell$. Therefore we obtain recursively $w_{n}=0$ also for all $n\geq\ell$. The proof of Lemma~\ref{lem:extend} is complete.
\end{proof}

Now we may prove the non-attainability. Suppose $\hat{u}\in H^{\ell}$ fulfills~\eqref{eq:hrb_ineq_thm2} as equality whose (both) sides are finite. As we assume~\eqref{eq:assum_A1}, \eqref{eq:assum_A2}, and~\eqref{eq:assum_A3''}, Lemma~\ref{lem:extend} applies and therefore $\mathcal{R}_{k}^{(\ell)}(\gg;\hat{u})=0$ for all $0\leq k<\ell$. In particular, for $k=0$, we 
have
\[
\frac{(-\Delta)^{\ell-1}\div\gg_{n}}{\div\gg_{n}}\left|\sqrt{\frac{\gg_{n}}{\gg_{n+1}}}\, \hat{u}_{n+1}-\sqrt{\frac{\gg_{n+1}}{\gg_{n}}}\, \hat{u}_{n}\right|^{2}=0
\]
for all $n\geq\ell$, see~\eqref{eq:def_remainder_R}. If $\ell=1$, the prefactor on the left-hand side equals $1$. If $\ell\geq2$, the prefactor is strictly positive by assumption~\eqref{eq:assum_A3''} and also \eqref{eq:assum_A1}. In any case, $\hat{u}$ is a solution of the first-order difference equation
\[
\sqrt{\frac{\gg_{n}}{\gg_{n+1}}}\, \hat{u}_{n+1}-\sqrt{\frac{\gg_{n+1}}{\gg_{n}}}\, \hat{u}_{n}=0
\]
for all $n\geq\ell$. Solving the equation, we find that $\hat{u}_{n}=c\gg_{n}$ for all $n\geq\ell$, where $c$ is a complex constant. Consequently, using expansion~\eqref{eq:assum_asympt_g}, we obtain
\[
 \hat{u}_{n}=c\alpha_{0} n^{\ell-s}+\bigO\left(n^{\ell-3/2}\right), \; \mbox{ for } n\to\infty.
\]
Further, similarly as in~\eqref{eq:rho_lower_bound_inproof}, we find that
\[
 \rho_{n}(\gg)=\frac{\gamma}{n^{2\ell}}+\bigO\left(\frac{1}{n^{2\ell+1}}\right), \; \mbox{ as } n\to\infty,
\]
with a constant $\gamma\neq0$ (a concrete computation yields $\gamma=(s)_{\ell}(1-s)_{\ell}$, cf.~\eqref{eq:rho_q_asympt} below). Therefore
\[
 \rho_{n}(\gg)|\hat{u}_{n}|^{2}=\frac{|c|^{2}\alpha_{0}^{2}\gamma}{n^{2s}}\left[1+\bigO\left(\frac{1}{n}\right)\right], \;\mbox{ for } n\to\infty.
\]
By our assumptions, $\alpha_{0}\neq0$, $\gamma\neq0$, $s\leq1/2$, and 
\[
 \sum_{n=\ell}^{\infty}\rho_{n}(\gg)|\hat{u}_{n}|^{2}<\infty,
\]
from which we infer that $c=0$, i.e. $\hat{u}=0$. The proof of Theorem~\ref{thm:3} is complete.

\qed

\subsection{Proof of Theorem~\ref{thm:4}}\label{subsec:thm4_proof}

For the parameter sequence $\gg^{(\ell)}$ defined by~\eqref{eq:def_gg_l}, we verify that
\begin{enumerate}[a)]
\item $\div^{k}\gg_{n}^{(\ell)}>0$ for all $n\geq\ell-k$ and $0\leq k<\ell$;
\item $(-\Delta)^{\ell-k}\div^{k}\gg_{n}^{(\ell)}>0$ for all $n\geq\ell-k$ and $0\leq k<\ell$.
\end{enumerate}
Claim~(a) together with the obvious fact $\gg^{(\ell)}\in H^{\ell}$ means that $\gg^{(\ell)}$ fulfills assumption~\eqref{eq:assum_A1}. Claim~(b) together with the definition~\eqref{eq:def_gg_l} of $\gg^{(\ell)}$ implies that $\gg^{(\ell)}$ satisfies also assumptions~\eqref{eq:assum_A2}, \eqref{eq:assum_A2'}, and~\eqref{eq:assum_A3''} with $s=1/2$. Therefore $\rho^{(\ell)}=(-\Delta)^{\ell}\gg^{(\ell)}/\gg^{(\ell)}$ is an optimal strictly positive discrete Hardy--Rellich--Birman weight by Theorems~\ref{thm:2} and~\ref{thm:3}.

a) We verify claim~(a). First we show that $\div^{\ell}\gg^{(\ell)}_{n}>0$ for all $n\geq0$. For $x\geq0$, let us denote
\[
 \gg^{(\ell)}(x):=\sqrt{x}(x-1)\dots(x-\ell+1).
\]
Then $\gg^{(\ell)}_{n}=\gg^{(\ell)}(n)$ for all $n\in\N_{0}$.
By~\cite[Eq.~(26.8.7)]{dlmf}, we have 
\begin{equation}
 \gg^{(\ell)}(x)=\sum_{j=1}^{\ell}s(\ell,j)x^{j-1/2},
\label{eq:g_n_polyn_stirling_1st}
\end{equation}
where $s(\ell,j)$ are the Stirling numbers of the first kind~\eqref{eq:def_stirling_1st}.
Notice that $(-1)^{\ell+j}s(\ell,j)>0$ for all $1\leq j\leq\ell$.
By Lemma~\ref{lem:opt2}, for any $n\in\N_{0}$ there exists $\xi\in(n,n+\ell)$ such that  
\[
 \div^{\ell}\gg_{n}^{(\ell)}=\frac{\dd^{\ell}\gg^{(\ell)}}{\dd x^{\ell}}(\xi).
\]
Since for any $x>0$, we have 
\[
\frac{\dd^{\ell}\gg^{(\ell)}}{\dd x^{\ell}}(x)=\sum_{j=1}^{\ell}s(\ell,j)\frac{\dd^{\ell}}{\dd x^{\ell}}x^{j-1/2}=\sum_{j=1}^{\ell}b_{j}^{(\ell)}x^{j-1/2-\ell},
\]
where
\[
 b_{j}^{(\ell)}:=(-1)^{\ell+j}s(\ell,j)\prod_{i=1}^{\ell}\left|j+\frac{1}{2}-i\right|>0
\]
for all $1\leq j\leq\ell$, we see that $\div^{\ell}\gg_{n}^{(\ell)}>0$ for all $n\geq0$, indeed.

Next, notice that, since $\gg^{(\ell)}\in H^{\ell}$, we have 
\begin{equation}
 \div^{k}\gg^{(\ell)}_{\ell-k}=\gg_{\ell}^{(\ell)}=\sqrt{\ell}(\ell-1)!>0
\label{eq:div_k_pos_inproof}
\end{equation}
for every $0\leq k<\ell$.

By definition of the discrete divergence, $\div^{\ell}\gg^{(\ell)}>0$ on $\N_{0}$ means that the sequence $\div^{\ell-1}\gg^{(\ell)}$ is strictly increasing on $\N_{0}$. Since $\div^{\ell-1}\gg^{(\ell)}_{1}>0$ by~\eqref{eq:div_k_pos_inproof}, we conclude that $\div^{\ell-1}\gg^{(\ell)}_{n}>0$ for all $n\geq1$. Iterating this argument, we verify claim~(a).

b) We verify claim~(b). We make use of Lemma~\ref{lem:opt2} once more. Since
\[
 (-\Delta)^{\ell-k}\div^{k}\gg_{n}^{(\ell)}=(-1)^{\ell-k}\div^{2\ell-k}\gg^{(\ell)}_{n-\ell+k},
\]
Lemma~\ref{lem:opt2} implies that, for any $n\geq\ell-k$ and $0\leq k<\ell$, there is $\xi>0$ such that
\[
(-\Delta)^{\ell-k}\div^{k}\gg_{n}^{(\ell)}=(-1)^{\ell-k}\,\frac{\dd^{2\ell-k}\gg^{(\ell)}}{\dd x^{2\ell-k}}(\xi).
\]
Similarly as in part~(a), we find this time that
\[
(-1)^{\ell-k}\,\frac{\dd^{2\ell-k}\gg^{(\ell)}}{\dd x^{2\ell-k}}(x)=\sum_{j=1}^{\ell}c_{j}^{(\ell)}x^{j-1/2-2\ell+k},
\]
where
\[
 c_{j}^{(\ell)}:=(-1)^{\ell+j}s(\ell,j)\prod_{i=1}^{2\ell-k}\left|j+\frac{1}{2}-i\right|>0
\]
for every $1\leq j \leq\ell$ and $x>0$. Consequently, 
\[
(-1)^{\ell-k}\,\frac{\dd^{2\ell-k}\gg^{(\ell)}}{\dd x^{2\ell-k}}(x)>0
\]
for all $x>0$ and $0\leq k<\ell$, and the claim~(b) follows. The proof of Theorem~\ref{thm:4} is complete.

\qed

\subsection{Proof of Theorem~\ref{thm:5}}\label{subsec:thm5_proof}

First, with the aid of the generalized binomial theorem and~\eqref{eq:def_X_ml}, we find for all $\nu\in\R$ and $n>\ell$ that
\begin{align*}
 (-\Delta)^{\ell}n^{\nu}&=\sum_{j=-\ell}^{\ell}\binom{2\ell}{\ell+j}(-1)^{j}(n+j)^{\nu}=
 n^{\nu}\sum_{j=-\ell}^{\ell}\binom{2\ell}{\ell+j}(-1)^{j}\left(1+\frac{j}{n}\right)^{\nu}\\
 &= n^{\nu}\sum_{j=-\ell}^{\ell}\binom{2\ell}{\ell+j}(-1)^{j}\sum_{m=0}^{\infty}\binom{\nu}{m}\frac{j^{m}}{n^{m}}=n^{\nu}\sum_{m=0}^{\infty}\binom{\nu}{m}\frac{X_{m}^{(\ell)}}{n^{m}}.
\end{align*}
Recalling~\eqref{eq:X_vanish_id}, we arrive at the identity
\begin{equation}
(-\Delta)^{\ell}n^{\nu}=\sum_{m=2\ell}^{\infty}\binom{\nu}{m}\frac{X_{m}^{(\ell)}}{n^{m-\nu}}
\label{eq:Delta_monial_id_inproof}
\end{equation}
for all $\nu\in\R$ and $n>\ell$. If $\nu>0$, the convergence of the series in~\eqref{eq:Delta_monial_id_inproof} can be extended to all $n\geq\ell$ by inspection of the asymptotic behavior of the summand. Namely, one deduces from~\eqref{eq:def_X_ml} and the Stirling formula that 
\[
 X_{2m}^{(\ell)}\sim2(-1)^{\ell}\ell^{2m} \quad\mbox{ and }\quad \binom{\nu}{m}\sim\frac{1}{\Gamma(-\nu)}\frac{(-1)^{m}}{m^{\nu+1}}
\]
as $m\to\infty$. Therefore the non-vanishing even summands of~\eqref{eq:Delta_monial_id_inproof} behave as
\[
\binom{\nu}{2m}\frac{X_{2m}^{(\ell)}}{n^{2m-\nu}}\sim\frac{2^{-\nu}n^{\nu}}{\Gamma(-\nu)}\frac{(-1)^{\ell}}{m^{\nu+1}}\left(\frac{\ell}{n}\right)^{2m}
\]
for $m\to\infty$. Consequently, the expansion~\eqref{eq:Delta_monial_id_inproof} remains convergent also for $n=\ell$, if $\nu>0$.

i) We prove claim~(i).
By using~\eqref{eq:g_n_polyn_stirling_1st} together with~\eqref{eq:Delta_monial_id_inproof}, we find that
\begin{align*}
(-\Delta)^{\ell}\gg_{n}^{(\ell)}&=\sum_{j=1}^{\ell}s(\ell,j)(-\Delta)^{\ell}n^{j-1/2}=\sum_{j=1}^{\ell}s(\ell,j)\sum_{m=2\ell}^{\infty}\binom{j-1/2}{m}\frac{X_{m}^{(\ell)}}{n^{m-j+1/2}}\\
 &=\sum_{m=2\ell}^{\infty}\sum_{j=0}^{\ell-1}\binom{\ell-j-1/2}{m}s(\ell,\ell-j)\frac{X_{m}^{(\ell)}}{n^{m-\ell+j+1/2}}
\end{align*}
for all $n\geq\ell$. It follows that
\begin{align*}
\frac{(-\Delta)^{\ell}\gg_{n}^{(\ell)}}{\gg_{n}^{(\ell)}}&=\frac{n^{\ell-1}}{(n-1)\dots(n-\ell+1)}\sum_{m=2\ell}^{\infty}\sum_{j=0}^{\ell-1}\binom{\ell-j-1/2}{m}s(\ell,\ell-j)\frac{X_{m}^{(\ell)}}{n^{m+j}}\\
&=\frac{n^{\ell-1}}{(n-1)\dots(n-\ell+1)}\sum_{k=2\ell}^{\infty}\left[\,\sum_{m=2\ell}^{k}\binom{\ell+m-k-1/2}{m}s(\ell,\ell+m-k)X_{m}^{(\ell)}\right]\frac{1}{n^{k}},
\end{align*}
from which we extract formula~\eqref{eq:coeff_r_kl} for coefficients $r_{k}^{(\ell)}$.

Next, we compute the first two coefficients $r_{2\ell}^{(\ell)}$ and $r_{2\ell+1}^{(\ell)}$. For $k=2\ell$, formula~\eqref{eq:coeff_r_kl} yields
\[
 r_{2\ell}^{(\ell)}=\binom{\ell-1/2}{2\ell}s(\ell,\ell)X_{2\ell}^{(\ell)}=\left(\frac{1}{2}\right)_{\ell}^{2}
\]
since $s(\ell,\ell)=1$ and $X_{2\ell}^{(\ell)}=(-1)^{\ell}(2\ell)!$. Similarly, putting $k=2\ell+1$ in~\eqref{eq:coeff_r_kl} and taking into account that $X_{2\ell+1}^{(\ell)}=0$, we find that
\[
 r_{2\ell+1}^{(\ell)}=\binom{\ell-3/2}{2\ell}s(\ell,\ell-1)X_{2\ell}^{(\ell)} 
 =\frac{\ell(\ell-1)(2l+1)}{2(2l-1)}\left(\frac{1}{2}\right)_{\ell}^{2},
\]
where we have used that $s(\ell,\ell-1)=-\ell(\ell-1)/2$. The proof of claim~(i) of Theorem~\ref{thm:5} is complete.

ii) We prove claim~(ii). The claim for $\ell=1$ is an immediate consequence of the known explicit formulas
\begin{equation}
 r_{2k+1}^{(1)}=0 \quad\mbox{ and }\quad r_{2k}^{(1)}=\frac{1}{2^{4k-1}(4k-1)}\binom{4k}{2k}, \quad k\geq1.
\label{eq:r_k^1}
\end{equation}

Suppose $\ell\geq2$. Recalling formulas~\eqref{eq:def_stirling_1st} and~\eqref{eq:X_id_pos}, we find by inspection of the sign of each of the three terms in the sum from~\eqref{eq:coeff_r_kl} that
\[
 (-1)^{\ell}X_{m}^{(\ell)}\geq0, \quad (-1)^{m+k}s(\ell,\ell+m-k)\geq0, \quad (-1)^{\ell+k}\binom{\ell+m-k-1/2}{m}\geq0
\]
for all $k\geq m \geq2\ell$. Taking also into account that $X_{m}^{(\ell)}=0$ if $m$ is odd, we see that each summand from the sum for coefficients $r_{k}^{(\ell)}$ in~\eqref{eq:coeff_r_kl} is nonnegative. Consequently, $r_{k}^{(\ell)}\geq0$ for all $k\geq2\ell$. Moreover, we can estimate $r_{k}^{(\ell)}$ from below by the last non-vanishing summand which corresponds to index $m=k$ if $k$ is even and $m=k-1$ if $k$ is odd.

If $k\geq2\ell$ is even, then
\[
r_{k}^{(\ell)}\geq\binom{\ell-1/2}{k}s(\ell,\ell)X_{k}^{(\ell)}=\left|\binom{\ell-1/2}{k}X_{k}^{(\ell)}\right|>0,
\]
by~\eqref{eq:X_id_pos}. If $k\geq2\ell$ is odd, then
\[
r_{k}^{(\ell)}\geq\binom{\ell-3/2}{k-1}s(\ell,\ell-1)X_{k-1}^{(\ell)}=\left|\binom{\ell-3/2}{k-1}\binom{\ell}{2}X_{k-1}^{(\ell)}\right|>0,
\]
by~\eqref{eq:X_id_pos} again. In total, we verify that $r_{k}^{(\ell)}>0$ for all $k\geq2\ell$ and the proof of claim (ii) of Theorem~\ref{thm:5} is complete.

iii) We prove claim~(iii). Let $n\geq\ell\geq2$. The inequality 
\[
\frac{(-\Delta)^{\ell} n^{\ell-1/2}}{n^{\ell-1/2}}>\left(\frac{1}{2}\right)_{\ell}^{2}\frac{1}{n^{2\ell}}
\]
has been already proven in~\cite{ger-kre-sta_23}. Alternatively, we can deduce it by using~\eqref{eq:Delta_monial_id_inproof} with $\nu=\ell-1/2$, noticing that each summand corresponding to $m$ odd is vanishing while each summand corresponding to $m$ even is positive. Then we readily estimate 
\[
\frac{(-\Delta)^{\ell} n^{\ell-1/2}}{n^{\ell-1/2}}>\binom{\ell-1/2}{2\ell}\frac{X_{2\ell}^{(\ell)}}{n^{2\ell}}=\left(\frac{1}{2}\right)_{\ell}^{2}\frac{1}{n^{2\ell}}
\]
by~\eqref{eq:X_id_pos}. 

Next, we verify the inequality 
\[
\rho_{n}^{(\ell)}>\frac{(-\Delta)^{\ell} n^{\ell-1/2}}{n^{\ell-1/2}}
\]
for $n\geq\ell\geq2$. We show that the summands in 
\[
 (-\Delta)^{\ell}\gg_{n}^{(\ell)}=\sum_{j=1}^{\ell}s(\ell,j)(-\Delta)^{\ell}n^{j-1/2}
\]
are all positive. First, recall that $(-1)^{j+\ell}s(j,\ell)>0$ for all $j=1,\dots,\ell$, see~\eqref{eq:def_stirling_1st}. Second, Lemma~\ref{lem:opt2} implies that there exists $\xi\in(n-\ell,n+\ell)$, hence $\xi>0$, such that 
\[
(-\Delta)^{\ell}n^{j-1/2}=(-1)^{\ell}\frac{\dd^{2\ell}}{\dd x^{2\ell}}\bigg|_{x=\xi}\hskip-8pt x^{j-1/2}=(-1)^{j+\ell}\,\xi^{j-2\ell-1/2}\prod_{i=1}^{2\ell}\left|j+\frac{1}{2}-i\right|.
\]
It follows that $(-1)^{j+\ell}\Delta^{\ell}n^{j-1/2}>0$ for all $j=1,\dots,\ell$. Consequently, we may estimate
\[
 (-\Delta)^{\ell}\gg_{n}^{(\ell)}>s(\ell,\ell)(-\Delta)^{\ell}n^{\ell-1/2}=(-\Delta)^{\ell}n^{\ell-1/2},
\]
from which we find that 
\[
\frac{(-\Delta)^{\ell}\gg_{n}^{(\ell)}}{\gg_{n}^{(\ell)}}>\frac{n^{\ell-1}}{(n-1)\dots(n-\ell+1)}\frac{(-\Delta)^{\ell}n^{\ell-1/2}}{n^{\ell-1/2}}>\frac{(-\Delta)^{\ell}n^{\ell-1/2}}{n^{\ell-1/2}}.
\]
The proof of Theorem~\ref{thm:5} is complete.
\qed

\section{More general families of Hardy--Rellich--Birman weights}\label{sec:gen_weights}

The concrete parameter sequence $\gg^{(\ell)}$ defined by~\eqref{eq:def_gg_l} has been used because the corresponding weight $\rho^{(\ell)}$ is optimal and still relatively simple. But this is not the only optimal discrete Hardy--Rellich--Birman weight. In this section, we briefly discuss various generalizations adding one or more parameters and discuss optimality of the resulting generalized weights as an application of abstract Theorems~\ref{thm:1}--\ref{thm:3}.

\subsection{A countable family of Hardy--Rellich--Birman weights}

Recall that the concrete parameter sequence $\gg=\gg^{(\ell)}$ defined by~\eqref{eq:def_gg_l} meets all the assumptions of Theorems~\ref{thm:1}--\ref{thm:3} for all $\ell\in\N$. In fact, any such a family of parameter sequences gives rise to a denumerable number of new discrete Hardy--Rellich--Birman weights that are also optimal. In the next statement, we use the notation $\gg^{[\ell]}$ to designate explicitly a dependence of the parameter sequence on $\ell$ but distinguish from the concrete parameter sequence~\eqref{eq:def_gg_l} by using the square brackets.

\begin{thm}\label{thm:rho_m_weights}
If, for all $\ell\in\N$, $\gg^{[\ell]}$ fulfills assumptions \eqref{eq:assum_A1}, \eqref{eq:assum_A2}, and~\eqref{eq:assum_A2'}, then $\rho^{[\ell,m]}$, defined by
\[
 \rho^{[\ell,m]}_{n}:=\frac{(-\Delta)^{\ell}\div^{m}\gg^{[\ell+m]}_{n}}{\div^{m}\gg^{[\ell+m]}_{n}}
\]
for $n\geq\ell$, is a discrete Hardy--Rellich--Birman weight for all $m\in\N_{0}$, i.e. $(-\Delta)^{\ell}\geq\rho^{[\ell,m]}\geq0$. If, in addition, for all $\ell\in\N$, $\gg^{[\ell]}$ satisfies also~\eqref{eq:assum_A3}, \eqref{eq:assum_A3'}, and \eqref{eq:assum_A3''}, then $\rho^{[\ell,m]}$ is, for all $m\in\N_{0}$, critical, optimal near infinity, and non-attainable, respectively.
\end{thm}

\begin{proof}
\emph{Step 1:} Suppose $\gg^{[\ell]}$ satisfies ~\eqref{eq:assum_A1}, \eqref{eq:assum_A2}, and \eqref{eq:assum_A2'} for every $\ell\in\N$.
One readily checks that the assumption~\eqref{eq:assum_A1} for $\gg^{[\ell+1]}$, where $\ell$ is replaced by $\ell+1$, implies that $\div\gg^{[\ell+1]}$ satisfies~\eqref{eq:assum_A1}. By induction, we find that $\div^{m}\gg^{[\ell+m]}$ satisfies~\eqref{eq:assum_A1} for all $m\in\N$.
Analogously, inequalities of assumption~\eqref{eq:assum_A2} for $\gg^{[\ell+1]}$ includes the respective inequalities of~\eqref{eq:assum_A2} for $\div\gg^{[\ell+1]}$, and hence~\eqref{eq:assum_A2} holds for $\div^{m}\gg^{[\ell+m]}$ for all $m\in\N$.

Further, assumption~\eqref{eq:assum_A2} for $\gg^{[\ell+1]}$ with $k=1$ yields inequalities 
\begin{equation}
 (-\Delta)^{\ell}\div\gg^{[\ell+1]}_{n}\geq0, \quad \forall n\geq\ell+1.
\label{eq:a2'_aux_inproof}
\end{equation}
In order to deduce it also for $n=\ell$, and hence to check the assumption \eqref{eq:assum_A2'} is fulfilled for $\div\gg^{[\ell+1]}$, we need to apply \eqref{eq:assum_A2'} to $\gg^{[\ell+1]}$ which can be written as the inequality
\[
 -\grad(-\Delta)^{\ell}\div\gg^{[\ell+1]}_{n}\geq0
\]
for all $n\geq\ell+1$. It follows that, for all $n\geq\ell+1$, we have
\[
 (-\Delta)^{\ell}\div\gg^{[\ell+1]}_{n-1}\geq (-\Delta)^{\ell}\div\gg^{[\ell+1]}_{n},
\]
which together with~\eqref{eq:a2'_aux_inproof} implies the inequality of~\eqref{eq:a2'_aux_inproof} holds also with $n=\ell$. Thus, $\div\gg^{[\ell+1]}$ satisfies~\eqref{eq:assum_A2'} and, by induction, $\div^{m}\gg^{[\ell+m]}$ satisfies~\eqref{eq:assum_A2'} for all $m\in\N$.
In total, we have shown that, for any $m\in\N$, $\div^{m}\gg^{[\ell+m]}$ fulfills~\eqref{eq:assum_A1}, \eqref{eq:assum_A2}, and \eqref{eq:assum_A2'}, therefore  $\rho^{[\ell,m]}$ is a~discrete Hardy--Rellich--Birman weight by Theorem~\ref{thm:2}.

\emph{Step 2:} It is easy to see that, if $\gg^{[\ell+1]}$ admits the expansion~\eqref{eq:assum_asympt_g} with $\ell$ replaced by $\ell+1$, $\alpha_{0}\neq0$, and $s\in(0,1)$, then $\div\gg^{[\ell+1]}$ fulfills~\eqref{eq:assum_asympt_g} with the same $s$ and $\alpha_{0}$ replaced by $(\ell+1-s)\alpha_{0}\neq0$.
Consequently, from assumptions \eqref{eq:assum_A3}, \eqref{eq:assum_A3'} or the asymptotic expansion part of~\eqref{eq:assum_A3''} satisfied by $\gg^{[\ell+1]}$, one deduces the respective conditions to hold for $\div\gg^{[\ell+1]}$. By induction, we extend the claims to $\div^{m}\gg^{[\ell+m]}$ for all $m\in\N$. Consequently, providing $\gg^{[\ell]}$ to fulfill~\eqref{eq:assum_A3} and~\eqref{eq:assum_A3'}, $\rho^{[\ell,m]}$ is critical and optimal near infinity, respectively, by Theorem~\ref{thm:3}.

Suppose finally that $\gg^{[\ell]}$ satisfies also the strict inequalities from~\eqref{eq:assum_A3''} for all $\ell\in\N$. Again, we verify that also $\div\gg^{[\ell+1]}$ fulfills the same inequalities, and hence the assumption~\eqref{eq:assum_A3''}. By using induction and Theorem~\ref{thm:3}, we then conclude that $\rho^{[\ell,m]}$ is non-attainable for all $m\in\N$.

First, when~\eqref{eq:assum_A3''} is imposed on $\gg^{[\ell+1]}$, we get the inequalities 
\begin{equation}
 (-\Delta)^{\ell+1}\gg^{[\ell+1]}_{n}>0 \;\mbox{ and }\; 
 (-\Delta)^{\ell}\div\gg^{[\ell+1]}_{n}>0,
 \quad \forall n\geq \ell+1.
\label{eq:a3'_aux_inproof}
\end{equation}
The first inequality yields
\[
-\grad(-\Delta)^{\ell}\div \gg^{[\ell+1]}_{n}>0,
\]
i.e. $(-\Delta)^{\ell}\div \gg^{[\ell+1]}_{n-1}>(-\Delta)^{\ell}\div \gg^{[\ell+1]}_{n}$ for all $n\geq\ell+1$, which implies the second inequality in~\eqref{eq:a3'_aux_inproof} must hold also for $n=\ell$.

It remains to verify that $(-\Delta)^{\ell-1}\div^{2}\gg^{[\ell+1]}_{n}>0$ for all $n\geq\ell$ assuming $\ell\geq2$ because, for $\ell=1$, the second inequality condition from~\eqref{eq:assum_A3''} is void.
Assumption~\eqref{eq:assum_A2} applied to $\gg^{[\ell+1]}$ with $k=2$ yields 
\begin{equation}
(-\Delta)^{\ell-1}\div^{2}\gg^{[\ell+1]}_{n}\geq0
\label{eq:tow_aux_inproof}
\end{equation}
for all $n\geq\ell$. We want to show that inequalities~\eqref{eq:tow_aux_inproof} are actually all strict. Suppose that there exists $n_{0}\geq\ell$ such that $(-\Delta)^{\ell-1}\div^{2}\gg^{[\ell+1]}_{n_0}=0$. The second inequality of~\eqref{eq:a3'_aux_inproof} tells us that
\[
 (-\Delta)^{\ell-1}\div^{2}\gg^{[\ell+1]}_{n}<(-\Delta)^{\ell-1}\div^{2}\gg^{[\ell+1]}_{n-1}, \quad\forall n\geq\ell+1.
\]
When combined with our assumption, it follows that $(-\Delta)^{\ell-1}\div^{2}\gg^{[\ell+1]}_{n}<0$ for all $n>n_{0}$, contradicting~\eqref{eq:tow_aux_inproof}. The proof of Theorem~\ref{thm:rho_m_weights} is complete.
\end{proof}

\begin{rem}
Theorem~\ref{thm:rho_m_weights} is applicable to the parameter sequence~\eqref{eq:def_gg_l}, therefore the corresponding weights $\rho^{[\ell,m]}$ are optimal strictly positive discrete Hardy--Rellich--Birman weights for all $m\in\N_{0}$ and $\ell\in\N$. Clearly, $\rho^{[\ell,0]}$ coincides with~\eqref{eq:def_rho_l}. In particular, for $\ell=1$, we get a sequence of optimal Hardy weights $\rho^{[1,m]}$, $m\in\N_{0}$.  For $m=0$, $\rho^{[1,0]}$ is the Keller--Pinchover--Pogorselski  weight,
\[
 \rho_{n}^{[1,0]}=2-\frac{\sqrt{n+1}+\sqrt{n-1}}{\sqrt{n}}=\frac{1}{4n^{2}}+\bigO\left(\frac{1}{n^{4}}\right).
\]
For example, if $m=1$, we get a new optimal discrete Hardy weight
\[
 \rho_{n}^{[1,1]}=2-\frac{(n+1)\sqrt{n+2}-n\sqrt{n+1}-(n-1)\sqrt{n}+(n-2)\sqrt{n-1}}{n\sqrt{n+1}-(n-1)\sqrt{n}}
\]
of asymptotically heavier tail than $\rho^{[1,0]}$, for 
\[ 
 \rho^{[1,1]}_{n}=\frac{1}{4n^{2}}+\frac{1}{12n^{3}}+\bigO\left(\frac{1}{n^{4}}\right), \quad n\to\infty.
\]
\end{rem}

\subsection{A $q$-generalization of $\rho^{(\ell)}$} For a parameter $q>0$, we consider 
\begin{equation}
 \gg_{n}^{(\ell)}(q):=n^{q}\prod_{j=1}^{\ell-1}(n-j) 
 \quad\mbox{ and }\quad
 \rho^{(\ell)}(q):=\frac{(-\Delta)^{\ell}\gg^{(\ell)}(q)}{\gg^{(\ell)}(q)}.
\label{eq:def_g_rho_q}
\end{equation}
In the Hardy case $\ell=1$, weight $\rho^{(1)}(q)$ appeared already in~\cite{kre-lap-sta_jlms22}. Clearly, $\gg^{(\ell)}$ defined in~\eqref{eq:def_gg_l} corresponds to $q=1/2$. One can show that, if $q>1$, weight $\rho^{(\ell)}(q)$ is not nonnegative. Moreover, if $q=1$, $\rho^{(\ell)}(1)\equiv0$ for all $\ell\geq1$. Therefore we restrict ourselves to $q\in(0,1)$.

\begin{prop}\label{prop:q-gener}
Let $\ell\in\N$ and $q\in(0,1)$. Then $\rho^{(\ell)}(q)$ defined by~\eqref{eq:def_g_rho_q} is strictly positive discrete Hardy--Rellich--Birman weight. Furthermore, $\rho^{(\ell)}(q)$ is critical if and only if $q\in(0,1/2]$, non-attainable if and only if $q\in[1/2,1)$, and optimal near infinity if and only if $q=1/2$.
\end{prop}

\begin{proof}
For $q\in(0,1)$, claims (a) and (b) from the proof of Theorem~\ref{thm:4} in Section~\ref{subsec:thm4_proof} can be verified in an analogous fashion. Consequently, $\gg^{(\ell)}(q)$ meets assumptions~\eqref{eq:assum_A1}, \eqref{eq:assum_A2}, and~\eqref{eq:assum_A2'} and so $\rho^{(\ell)}(q)$ are discrete Hardy--Rellich--Birman weights for all $q\in(0,1)$ by Theorem~\ref{thm:2}.

Let us discuss the optimality of $\rho^{(\ell)}(q)$.\\
a) \emph{Criticality:} Suppose $q\in(0,1/2]$. Then assumption~\eqref{eq:assum_A3} holds for $\gg^{(\ell)}(q)$ with $s=1-q\geq 1/2$. Therefore $\rho^{(\ell)}(q)$ is critical by Theorem~\ref{thm:3}.

On the other hand, $\rho^{(\ell)}(q)$ is not critical for $q\in(1/2,1)$, which is a consequence of the non-trivial inequality
\begin{equation}
\rho_{n}^{(\ell)}(q)<\rho_{n}^{(\ell)}(1/2),
\label{eq:q-non-crit_inproof}
\end{equation}
that holds for all $n\geq\ell$ and $q\in(1/2,1)$.
We verify~\eqref{eq:q-non-crit_inproof}. First, using definition~\eqref{eq:def_g_rho_q} together with~\eqref{eq:g_n_polyn_stirling_1st}, one finds that~\eqref{eq:q-non-crit_inproof} is equivalent to the inequality 
\[
\sum_{j=1}^{\ell}s(\ell,j)(-\Delta)^{\ell}n^{j-1+q}<n^{q-1/2}\sum_{j=1}^{\ell}s(\ell,j)(-\Delta)^{\ell}n^{j-1/2}
\]
for all $n\geq\ell$. Recalling that $(-1)^{j+\ell}s(\ell,j)>0$ for all $1\leq j \leq\ell$ and $\ell\in\N$, see~\eqref{eq:def_stirling_1st}, it is sufficient to show that
\[
(-1)^{\ell+j}(-\Delta)^{\ell}n^{j-1+q}<(-1)^{\ell+j}n^{q-1/2}(-\Delta)^{\ell}n^{j-1/2}
\]
for all $1\leq j \leq\ell\leq n$. With the aid of expansion~\eqref{eq:Delta_monial_id_inproof} and the fact that coefficients $X_{m}^{(\ell)}$ therein vanish for $m$ odd, we may write the last inequality as
\[
 (-1)^{\ell+j}n^{j-1+q}\,\sum_{m=\ell}^{\infty}\binom{j-1+q}{2m}\frac{X_{2m}^{(\ell)}}{n^{2m}}
 < (-1)^{\ell+j}n^{j-1+q}\,\sum_{m=\ell}^{\infty}\binom{j-1/2}{2m}\frac{X_{2m}^{(\ell)}}{n^{2m}}.
\]
Bearing in mind that $(-1)^{\ell}X_{2m}^{(\ell)}>0$ for all $m\geq\ell$, see~\eqref{eq:X_id_pos}, the last inequality is established once we show that
\[
(-1)^{j}\binom{j-1+q}{2m}
 < (-1)^{j}\binom{j-1/2}{2m}
\]
for all $1\leq j \leq\ell\leq m$ and $q\in(1/2,1)$. But this can be verified easily with the aid of the elementary inequality $(q+k)(1-q+k)<(k+1/2)^{2}$, which holds for all $k\in\N_{0}$ and $q\neq1/2$.

b) \emph{Non-attainability:} 
For $q\in[1/2,1)$, we have $s=1-q\leq1/2$ in~\eqref{eq:assum_A3''}. Also, the strict inequalities of~\eqref{eq:assum_A3''} hold as one can verify in an analogous fashion as claim~(b) of Subsection~\ref{subsec:thm4_proof}. Therefore $\rho^{(\ell)}(q)$ is non-attainable for $q\in[1/2,1)$ by Theorem~\ref{thm:3}.

Conversely, suppose $q\in(0,1/2)$. With the aid of~\eqref{eq:Delta_monial_id_inproof} and~\eqref{eq:X_id_pos}, we find that
\begin{equation}
 \rho_{n}^{(\ell)}(q)=\binom{\ell-1+q}{2\ell}\frac{X_{2\ell}^{(\ell)}}{n^{2\ell}}+\bigO\left(\frac{1}{n^{2\ell+1}}\right)=\frac{(q)_{\ell}(1-q)_{\ell}}{n^{2\ell}}+\bigO\left(\frac{1}{n^{2\ell+1}}\right)
\label{eq:rho_q_asympt}
\end{equation}
as $n\to\infty$. Taking also into account that $\gg_{n}^{(\ell)}(q)=n^{\ell-1+q}+\bigO(n^{\ell-2+q})$ for $n\to\infty$, we observe that 
\[
 \sum_{n=\ell}^{\infty}\rho_{n}^{(\ell)}(q)\left|\gg_{n}^{(\ell)}(q)\right|^{2}<\infty,
\]
provided that $q\in(0,1/2)$. Since Lemma~\ref{lem:extend} applies to $\gg^{(\ell)}(q)$, recalling also that $\mathcal{R}_{k}^{(\ell)}(\gg;u)=0$ if $u=\gg$, see~\eqref{eq:def_remainder_R}, we may substitute for $u=\gg^{(\ell)}(q)$ into~\eqref{eq:hrb_id_init} getting the equality
\[
\sum_{n=\lceil\ell/2\rceil}^{\infty}\left|(-\Delta)^{\ell/2}\gg^{(\ell)}_{n}(q)\right|^{2}=\sum_{n=\ell}^{\infty}\rho_{n}^{(\ell)}(q)\left|\gg_{n}^{(\ell)}(q)\right|^{2}.
\]
Thus, $\rho^{(\ell)}(q)$ is attainable for $q\in(0,1/2)$.

c) \emph{Optimality near infinity:}
For $q=1/2$, the optimality of $\rho^{(\ell)}\equiv\rho^{(\ell)}(1/2)$ is asserted in Theorem~\ref{thm:4}. The non-optimality near infinity of $\rho^{(\ell)}(q)$ for $q\neq1/2$ is a consequence of the fact that
the constant by the leading term in~\eqref{eq:rho_q_asympt} satisfies
\[
 (q)_{\ell}(1-q)_{\ell}<\left(\frac{1}{2}\right)_{\ell}^{2}
\]
for $q\neq1/2$. This can be seen from the definition of the Pochhammer symbol and the inequality $(q+k)(1-q+k)<(k+1/2)^{2}$, once again. Consequently, for $q\neq1/2$ fixed, we find $\varepsilon>0$ small, such that 
\[
(1+\varepsilon)\rho_{n}^{(\ell)}(q)\leq\rho_{n}^{(\ell)}
\]
for all $n$ sufficiently large. Then, for all $M\in\N$ sufficiently large and any $u\in\mathcal{H}_{0}^{M}$, we have
\[
\sum_{n=\lceil\ell/2\rceil}^{\infty}\left|(-\Delta)^{\ell/2}u_{n}\right|^{2}\geq\sum_{n=M}^{\infty}\rho_{n}^{(\ell)}|u_{n}|^{2}\geq(1+\varepsilon)\sum_{n=M}^{\infty}\rho_{n}^{(\ell)}(q)|u_{n}|^{2},
\]
contradicting~\eqref{eq:def_opt_near_inf}. The proof of Proposition~\ref{prop:q-gener} is complete.
\end{proof}

\begin{rem}
 Proposition~\ref{prop:q-gener} can be combined with Theorem~\ref{thm:rho_m_weights}. Then for any $m\in\N_{0}$ and $q\in(0,1)$, sequence $\rho^{(\ell,m)}(q)$ defined by
 \[
  \rho^{(\ell,m)}_{n}(q):=\frac{(-\Delta)^{\ell}\div^{m}\gg^{(\ell+m)}_{n}(q)}{\div^{m}\gg^{(\ell+m)}_{n}(q)}
 \]
for $n\geq\ell$, with $\gg^{(\ell)}(q)$ as in~\eqref{eq:def_g_rho_q}, is a discrete Hardy--Rellich--Birman weight. Moreover, $\rho^{(\ell,m)}(q)$ is critical if $q\in(0,1/2]$, non-attainable if $q\in[1/2,1)$, and optimal near infinity if $q=1/2$.
\end{rem}

\subsection{Multi-parameter families of optimal discrete Hardy--Rellich--Birman weights}

For $\ell\geq2$, more optimal weights generalizing~\eqref{eq:def_rho_l} in $(\ell-1)$-parameters can be found. The basic idea for their detection is reminiscent of the one developed in~\cite{kel-pin-pog_18b}, where the authors  relate Hardy weights ($\ell=1$) to positive harmonic functions. For $\ell\geq2$, we seek polyharmonic functions, i.e. solutions of the equation
\[
 (-\Delta)^{\ell}\mathfrak{h}_{n}=0, \quad \forall n\geq\ell,
\]
satisfying the boundary condition $\mathfrak{h}_{0}=\dots=\mathfrak{h}_{\ell-1}=0$, and then take $\gg:=\sqrt{\mathfrak{h}}$, provided that $\mathfrak{h}\geq0$, as a candidate for the parameter sequence. Up to a multiplicative constant, a~general solution $\mathfrak{h}$ of this problem can be expressed as
\begin{equation}
 \mathfrak{h}_{n}=\prod_{j=0}^{\ell-1}(n-j)\prod_{k=1}^{\ell-1}(n-\alpha_{k}),
\label{eq:def_h}
\end{equation}
where $\alpha_{1},\dots,\alpha_{\ell-1}\in\R$ are parameters. Notice that, if $\alpha_{k}=k$, $\sqrt{\mathfrak{h}}$ coincides with the optimal weight $\gg^{(\ell)}$ of Theorem~\ref{thm:4}. For general $\alpha_{1},\dots,\alpha_{\ell-1}$, however, the assumptions \eqref{eq:assum_A1}, \eqref{eq:assum_A2}, and \eqref{eq:assum_A2'} impose additional non-trivial conditions on the parameters and we find it difficult to express these restrictions in terms of the parameters $\alpha_{1},\dots,\alpha_{\ell-1}$ directly. Nevertheless, claims (a) and (b) of Subsection~\ref{subsec:thm4_proof} on $\gg^{(\ell)}$ and perturbation arguments imply that the set of admissible values of $\alpha_{1},\dots,\alpha_{\ell-1}$ contains other solutions than the one corresponding to the particular parameter sequence $\gg^{(\ell)}$. As far as the optimality is concerned, notice that the assumption~\eqref{eq:assum_A3'} holds for $\gg=\sqrt{\mathfrak{h}}$, with $\mathfrak{h}$ given by~\eqref{eq:def_h}. Therefore the resulting weight $\rho(\gg)$ is critical and optimal near infinity. The non-attainability of $\rho(\gg)$ is again a question of the additional restrictions of the parameters $\alpha_{1},\dots,\alpha_{\ell-1}$ guaranteeing the strict inequalities of~\eqref{eq:assum_A3''} to hold.

We illustrate the situation in the still relatively simple case $\ell=2$ when our candidate is
\begin{equation}
 \gg_{n}(\alpha):=\sqrt{n(n-1)(n-\alpha)}.
\label{eq:def_g_alp}
\end{equation}
Assumption~\eqref{eq:assum_A1} requires $\gg_{n+1}(\alpha)>\gg_{n}(\alpha)>0$ for all $n\geq2$. The positivity of $\gg_{n}(\alpha)$ for all $n\geq2$ induces the restriction $\alpha<2$ which is also sufficient for the monotonicity $\gg_{n+1}(\alpha)>\gg_{n}(\alpha)$ for all $n\geq2$. Assumptions \eqref{eq:assum_A2} and~\eqref{eq:assum_A2'} amount to inequalities $0\leq(-\Delta)\div\gg_{n}(\alpha)\leq(-\Delta)\div\gg_{n-1}(\alpha)$ for all $n\geq2$, from which only the second inequality introduces new restrictions on $\alpha$ since
\[
 (-\Delta)\div\gg_{n}(\alpha)=\frac{3}{8n^{3/2}}+\bigO\left(\frac{1}{n^{5/2}}\right), \;\mbox{ as } n\to\infty.
\]
Thus, the final range for $\alpha<2$ is determined by the requirement $\Delta^{2}\gg_{n}(\alpha)\geq0$ for all $n\geq2$. It seems difficult, however, to find a solution analytically. Nevertheless, numerically we get the approximate range $0.847\leq\alpha\leq1.307$ (a suitable CAS such as Wolfram Mathematica is capable of expressing the lower and upper bounds in radicals). With sharp inequalities in the final restriction on $\alpha$, also~\eqref{eq:assum_A3''} holds. Thus, we conclude that for any $\alpha$ approximately within the range 
\[
 0.847<\alpha<1.307,
\]
the weight $\Delta^{2}\gg(\alpha)/\gg(\alpha)$, with $\gg(\alpha)$ given by~\eqref{eq:def_g_alp}, is strictly positive optimal discrete Rellich weight.

\subsection*{Acknowledgment}
F.~{\v S} acknowledges the support of the EXPRO grant No.~20-17749X of the Czech Science Foundation.

\appendix
\section{Proofs of Lemmas~\ref{lem:opt2} and~\ref{lem:opt3}}

\subsection{Proof of Lemma~\ref{lem:opt2}}

Let $p$ be a polynomial of degree less or equal to $N$ such that $p_{n+j}=g_{n+j}$ for all $j=0,1,\dots,N$. Here and below, we use the notation $p_{n}:=p(n)$ and $g_{n}:=g(n)$. Then
\[
 \div^{N}g_{n}=\sum_{j=0}^{N}\binom{N}{j}(-1)^{N-j}g_{n+j}=\sum_{j=0}^{N}\binom{N}{j}(-1)^{N-j}p_{n+j}=\div^{N}p_{n}.
\]

Next, let us write $p(x)=\sum_{k=0}^{N}a_{k}x^{k}$, where $a_{k}\in\R$. Notice that, if the degree of $p$ is less or equal to $N$, then the polynomial $\div p(x):=p(x+1)-p(x)$ is of degree less or equal to $N-1$. Moreover, it is easy to check that $\div^{N}x^{k}=0$ for $k=0,1,\dots N-1$ and $\div^{N}x^{N}=N!$. Consequently, 
\[
\div^{N}p(x)=a_{N}N!
\]
for any $x\in\R$.

Since the function $f:=g-p$ vanishes at all points $n,n+1,\dots,n+N$, for every $j=0,1,\dots,N-1$, there exist $c_{j}\in(n+j,n+j+1)$ such that $f'(c_{j})=0$ by Rolle's theorem. By iteration of the application of Rolle's theorem, we prove the existence of $\xi\in(n,n+N)$ such that $f^{(N)}(\xi)=0$, i.e. $g^{(N)}(\xi)=p^{(N)}(\xi)$. 

In total, we have
\[
 \div^{N} g_{n}=\div^{N} p_{n}=a_{N}N!=p^{(N)}(\xi)=g^{(N)}(\xi).
\]
The proof of Lemma~\ref{lem:opt2} is complete.

\subsection{Proof of Lemma~\ref{lem:opt3}}

The proof proceeds by induction in $m\in\N_{0}$. The statement is obviously true for  $m=0$. For $m=1$, one readily verifies that
\[
 \div(uv)=(\SS u)\div v+(\div u)v=u(\div v)+(\div u)\SS v,
\]
which follows that
\[
\div(uv)=\frac{1}{2}\left((\SS u)\div v+(\div u)v+u(\div v)+(\div u)\SS v\right)=(\div u)\M v+(\M u)\div v.
\]

Next, we assume that the statement holds true for some $m\in\N_{0}$ and deduce the formula for $m+1$. Using the induction hypothesis and the above computation, we obtain
\begin{align*}
&\div^{m+1}(uv)=\div(\div^{m}(uv))=\div\left(\sum_{j=0}^{m}\binom{m}{j}\left(\div^{j}\M^{m-j}u\right)\left(\div^{m-j}\M^{j}v\right)\right)\\
&=\sum_{j=0}^{m}\binom{m}{j}\left[\left(\div^{j+1}\M^{m-j}u\right)\left(\div^{m-j}\M^{j+1}v\right)+\left(\div^{j}\M^{m-j+1}u\right)\left(\div^{m-j+1}\M^{j}v\right)\right]\\
&=\left(\div^{m+1}u\right)\left(\M^{m+1}v\right)+\left(\M^{m+1} u\right)\left(\div^{m+1}v\right)\\
&\hskip104pt+\sum_{j=1}^{m}\left[\binom{m}{j-1}+\binom{m}{j}\right]\left(\div^{j}\M^{m+1-j}u\right)\left(\div^{m+1-j}\M^{j}v\right)\\
&=\sum_{j=0}^{m+1}\binom{m+1}{j}\left(\div^{j}\M^{m+1-j}u\right)\left(\div^{m+1-j}\M^{j}v\right).
\end{align*}
The proof of Lemma~\ref{lem:opt3} is complete.

\bibliographystyle{acm}

\end{document}